\documentclass[11pt]{amsart}
\usepackage{amsmath, amssymb, array}
\usepackage[top=30truemm,bottom=30truemm,left=30truemm,right=30truemm]{geometry} 
\usepackage{mathtools}
\usepackage[all]{xy}
\usepackage{ascmac}
\usepackage{mathtools}
\usepackage{mathrsfs}
\usepackage{color}
\usepackage{xcolor}
\usepackage[colorlinks=true, linkcolor=blue, citecolor=blue]{hyperref}
\usepackage{amscd}
\usepackage{fancyhdr}
\usepackage{amsrefs}
\usepackage{cleveref}
\usepackage{chngcntr}
\usepackage{apptools}
\usepackage{comment}
\usepackage{listings}
\usepackage{booktabs}
\usepackage{multirow}
\usepackage{url}
\usepackage[shortlabels]{enumitem}

\newcommand{\Q}{\mathbb{Q}}

\DeclareMathOperator{\rank}{rank} %rank
\DeclareMathOperator{\GL}{GL}
\DeclareMathOperator{\OA}{OA}

\DeclareMathOperator{\GDD}{GDD}

\DeclareMathOperator{\PG}{PG}

\DeclareMathOperator{\PTE}{PTE}

\DeclareMathOperator{\Orb}{Orb}

\title{Combinatorial designs and the Prouhet--Tarry--Escott problem}

\author{Munenori Inagaki$^{*}$}
\email{2225063t@stu.kobe-u.ac.jp}
\address{${}^*$Faculty of Engineering, Kobe University, 1-1 Rokkodai, Nada, Kobe 657-8501, Japan}

\author{Hideki Matsumura$^{**}$}
\email{hmatsumura@tmu.ac.jp}
\address{${}^{**}$Graduate School of Science, Tokyo Metropolitan University, 1-1 Minami-Osawa, Hachioji-shi, Tokyo 192-0397, Japan}

\author{Masanori Sawa$^{*}$}
\email{sawa@people.kobe-u.ac.jp}
 \address{${}^*$Graduate School of System Informatics, Kobe University, 1-1 Rokkodai, Nada, Kobe 657-8501, Japan}

\author{Yukihiro Uchida$^{**}$}
\email{yuchida@tmu.ac.jp}
\address{${}^{**}$Graduate School of Science, Tokyo Metropolitan University, 1-1 Minami-Osawa, Hachioji-shi, Tokyo 192-0397, Japan}

 \thanks{This research is supported by KAKENHI 20K03517, KAKENHI 24KJ0183, KAKENHI 25K17234 and KAKENHI 50508182 of the Japan Society for the Promotion of Science (JSPS), and by the Early Support Program for Grant-in-Aid for Scientific Research (B) of Kobe University.
This work is supported by the Research Institute for Mathematical Sciences, an International Joint Usage/Research Center located in Kyoto University. } 



\date{\today} 

\theoremstyle{plain}
 \newtheorem{theorem}{Theorem}[section] 
  \crefname{theorem}{Theorem}{Theorems}
 \AddToHook{env/theorem/begin}{\crefalias{theorem}{theorem}}
 \newtheorem{proposition}[theorem]{Proposition}
 \crefname{proposition}{Proposition}{Propositions}
 \AddToHook{env/proposition/begin}{\crefalias{theorem}{proposition}}
 \newtheorem{lemma}[theorem]{Lemma}
 \crefname{lemma}{Lemma}{Lemmas}
 \AddToHook{env/lemma/begin}{\crefalias{theorem}{lemma}}
 \newtheorem{corollary}[theorem]{Corollary}
  \crefname{corollary}{Corollary}{Corollaries}
 \AddToHook{env/corollary/begin}{\crefalias{theorem}{corollary}}
 
   \crefname{conjecture}{Conjecture}{Conjectures}
 \AddToHook{env/conjecture/begin}{\crefalias{theorem}{conjecture}}
 
 \crefname{question}{Question}{Questions}
 \AddToHook{env/question/begin}{\crefalias{theorem}{question}}
 \newtheorem{problem}[theorem]{Problem}
   \crefname{problem}{Problem}{Problems}
 \AddToHook{env/problem/begin}{\crefalias{theorem}{problem}}
 
    \crefname{notation}{Notation}{Notations}
 \AddToHook{env/notation/begin}{\crefalias{theorem}{notation}}
\crefname{table}{Table}{Tables}

\theoremstyle{definition} 
 \newtheorem{definition}[theorem]{Definition}
  \crefname{definition}{Definition}{Definitions}
 \AddToHook{env/definition/begin}{\crefalias{theorem}{definition}}
 \newtheorem{example}[theorem]{Example}
   \crefname{example}{Example}{Examples}
 \AddToHook{env/example/begin}{\crefalias{theorem}{example}}
 \newtheorem{remark}[theorem]{Remark}
   \crefname{remark}{Remark}{Remarks}
 \AddToHook{env/remark/begin}{\crefalias{theorem}{remark}}
    
   \crefname{claim}{Claim}{Claims}
 \AddToHook{env/claim/begin}{\crefalias{theorem}{claim}}

\begin{document}

\maketitle

\begin{abstract}
This is the first paper that provides a systematic treatment of the $r$-dimensional PTE problem in additive number theory, abbreviated by $\PTE_r$, through its connection with combinatorial design theory, the branch of combinatorial mathematics that deals with finite set systems or arrangements with the ^^ balancedness' conditions. 
We first propose a combinatorial reconsideration of the definition of ^^ proper solution' introduced by Alpers and Tijdeman (2007), and then prove a fundamental lower bound for the size of such solutions.
We exhibit high-dimensional minimal solutions with respect to the fundamental bound, which inherently have the structure of distinctive block designs or orthogonal arrays (OAs).
Next, we develop a powerful method for constructing $\PTE_r$ solutions via various classes of combinatorial designs such as block designs and OAs.
Furthermore, we explore two dimension-lifting methods for constructing $\PTE_r$ solutions:
one is a combinatorial composition that produces $\PTE_r$ solutions by embedding lower-dimensional solutions into OAs with $r$ columns, and the other is a recursive technique in which a $\PTE_r$ solution is constructed by taking the Cartesian product of two lower-dimensional solutions.
It is emphasized that our results generalize many previous works, including a measure-theoretic construction by Lorentz (1949) and its geometric analog by Alpers and Tijdeman (2007), a key lemma in Jacroux's work (1995) on the construction of sets of integers with equal power sums, and the famous Borwein solution and its two-dimensional extension by Matsumura and Sawa (2025).
In addition, we prove a characterization theorem for ideal solutions of the $\PTE_1$ and discuss the connection with a curious phenomenon, called half-integer design, that is rarely reported in the combinatorial design theory or spherical design theory.
\end{abstract}

\tableofcontents

%%%%%%%%%%%%%%%%%%%%%%%%%%%%%%%%%%%%%%%%%%%%%%%%%%%%%%%%%%%%%
%%%%%%%%%%%%%%%%%%%%%%%%%%%%%%%%%%%%%%%%%%%%%%%%%%%%%%%%%%%%%
\section{Introduction} \label{sect:Intro}

The {\it $r$-dimensional PTE problem of degree $m$ and size $n$} (see \cite{Alpers-Tijdeman}), abbreviated by $\PTE_r$, asks whether there exists a disjoint pair of multisets
$A = \{ (a_{11}, \ldots, a_{1r}), \ldots, (a_{n1}, \ldots, a_{nr})\}$, $B = \{ (b_{11}, \ldots, b_{1r}), \ldots, (b_{n1}, \ldots, b_{nr})\} \subset \mathbb{Z}^r$, such that
\[
\sum_{i=1}^n a_{i1}^{k_1} \cdots a_{ir}^{k_r} = \sum_{i=1}^n b_{i1}^{k_1} \cdots b_{ir}^{k_r},
\]
where the summations are taken over all combinations of nonnegative integers $k_1, \ldots, k_r$ satisfying the condition $1 \leq k_1+ \cdots +k_r \leq m$.
For a solution $(A,B)$ of the above equations, we often use the notation $[A]=^n_m [B]$ or $[\mathbf{a}_1,\ldots,\mathbf{a}_n] =_m^n [\mathbf{b}_1,\ldots,\mathbf{b}_n]$, where $\mathbf{a}_i = (a_{i1}, \ldots, a_{ir})$ and $\mathbf{b}_i = (b_{i1}, \ldots, b_{ir})$.
It is well known (cf.~\cite[Chapter 11, Theorem 1]{Borwein2002},~\cite[Theorem 4.1.9]{Ghiglione}) that the inequality $n \ge m+1$ holds whenever $[A]=^n_m [B]$ holds; a solution with equality is called {\it ideal}.

The $\PTE_1$ is a classical problem in additive number theory, as exemplified by the Euler--Goldbach solution
\[
[a,b,c,a+b+c]=^4_2 [0,a+b,a+c,b+c].
\]
For a good introduction to the early history on the $\PTE_1$, we refer the reader to~\cite[\S24]{Dickson}.
Among recent contributions by many others, one of the most influential results is the Borwein ideal solution given by
\begin{multline} \label{Borwein}
 [\pm (2a+2b), \pm (-ab-b-a+3), \pm (ab-b-a-3)] \\
=_5^6 [\pm (2b-2a), \pm (ab-b+a+3), \pm (-ab-b+a-3)];
\end{multline}
see for example \cite[p.\ 88]{Borwein2002}.
The $\PTE_1$ also finds applications in several other topics, including the problem of determining whether the chromatic polynomial of a given non-chordal graph has only integer zeros (\cite{HL}), the problem of searching for large twin smooth integers for compact isogeny-based post-quantum protocols (\cite{CMN}), and so on.

Alpers and Tijdeman \cite{Alpers-Tijdeman} generalized the classical $\PTE_1$ to the $\PTE_{\ge2}$ and mainly studied $\PTE_2$, providing a powerful construction which we look into later, through a geometric approach based on the concept of {\it switching components} in the theory of discrete tomography. Afterwards, Ghiglione~\cite{Ghiglione} generalized their ideas and extensively studied discrete-tomographic constructions for the $\PTE_{\geq 3}$; see \cite[Theorems 4.3.17 and 4.3.20]{Ghiglione} for details.
On the other hand, Matsumura and Sawa~\cite{MS2025} investigated a new approach to the $\PTE_2$ via a certain configuration of points on ellipse, called {\it ellipsoidal design}, which was introduced by Pandey \cite{Pandey2022} as a generalization of {\it spherical design} in $\mathbb{R}^2$.
They provided a method for constructing $\PTE_2$ solutions from a disjoint pair of ellipsoidal designs (see \cite[Proposition 5.1]{MS2025}).

In the present paper, we mainly consider the $\PTE_{\ge 3}$ and develop novel connection with {\it combinatorial designs}, including {\it orthogonal arrays (OAs)}, {\it block $t$-designs}, and so on.
We refer the reader to \Cref{subsect:design} for the definition of such `{\it designs on discrete spaces}'.
As far as the authors know, the only known purely-combinatorial approach for constructing $\PTE_{\ge 3}$ solutions is based on two-colorings of the binary hypercube $\{0,1\}^r$ due to Alpers and Tijdeman \cite[Theorem 12]{Alpers-Tijdeman} (see also~\cite[p.171]{AL2015}), which originally goes back to a celebrated work by Lorentz \cite[Section 4.3]{Lorentz1949} in measure theory.
For some discrete tomographic extensions of Lorentz's work (which is not the subject of this paper), see Kuba--Herman \cite{KH1999} and references therein.
What is remarkable here is that such {\it Lorentz--Alpers--Tijdeman (LAT) construction} can also be interpreted as a partition of the trivial OA of strength $r$ into two OAs of strength $r-1$, as will be explained in \cref{thm:LAT0} and \cref{2col} later. 
Thus, a main aim of the present paper is to generalize the LAT construction for OAs of arbitrary sizes, and even for other classes of combinatorial designs.

This is the first paper that provides a systematic treatment of the $\PTE_r$ through the connection with combinatorial design theory.
Although combinatorial approaches to the $\PTE_r$ via hypercube partitions have appeared in the literature, none have systematically developed a combinatorial design-based approach.
\footnote{The problem of relating $\PTE_r$ to ^^ designs on discrete spaces' was first proposed by Akihiko Yukie, and is one of the motivations for this research.} 
This paper not only offers number-theoretic interpretations of various concepts and methods in combinatorial design theory, but also demonstrates how these methods can provide constructive approaches to the PTE problem. 

This paper is organized as follows: 
\Cref{sect:preliminary} first makes a quick review of basic facts on the $\PTE_r$ and then discusses how proper solutions should be defined from the combinatorial perspective.
The same section also provides a fundamental lower bound for the size of proper solutions of the $\PTE_r$, which we call the {\it combinatorial PTE inequality} (\cref{thm:NewBound}), and then introduces some important classes of combinatorial designs. 
\Cref{sect:minimal} provides a proof of \cref{thm:NewBound} and then exhibits  minimal solutions with respect to our inequality, which inherently have the structure of
distinctive $t$-designs or OA related to the LAT construction, as will be seen from~\cref{thm:tight2,ex:tight4,ex:tight4-2}.
Sections \ref{sect:design} through \ref{sect:curious} are the main body of this paper.
\Cref{sect:design} provides many criteria to directly construct $\PTE_r$ solutions by taking a disjoint pair of combinatorial designs (see~\cref{OAtoPTE,GDDtoPTE,TdesigntoPTE}); in particular, such a construction for OAs (\cref{OAtoPTE}) substantially generalizes the LAT construction (\cref{thm:LAT0}).
\Cref{sect:lifting} explores two types of dimension-lifting constructions, which we call {\it combinatorial array lifting} (\cref{thm:OAL1,thm:OAL3}) and {\it Cartesian product lifting}  (\cref{thm:cp_lift}), respectively.
The former produces $\PTE_r$ solutions by embedding $\PTE_1$ solutions into OAs with $r$ columns, generalizing the Borwein solution (\ref{Borwein}) and its two-dimensional extension by Matsumura and Sawa~\cite[Proposition 5.12]{MS2025} (see \cref{thm:GeneralizedBorwein1,rem:2DBorwein1}).
The latter is a recursive method in which a $\PTE_r$ solution is constructed by taking the Cartesian product of two lower-dimensional solutions, generalizing Jacroux's result \cite[Lemma 1]{Jacroux1995} concerning sets of integers with equal power sums (\cref{cor:jac}).
\Cref{sect:curious} establishes a characterization theorem for ideal solutions of the $\PTE_1$ (\cref{lin}) and then discusses the connection with a curious phenomenon, called {\it half-integer design} (\cref{t+1/2}), that is not normally observed in combinatorial design theory or spherical design theory (\cref{HID}). 
Finally, \Cref{sect:conclusion} closes this paper with concluding remarks and future works.

%%%%%%%%%%%%%%%%%%%%%%%%%%%%%%%%%%%%%%%%
%%%%%%%%%%%%%%%%%%%%%%%%%%%%%%%%%%%%%%%%
\section{Preliminaries} \label{sect:preliminary}

In \Cref{sect:PTEr}, we first review basic facts on the $\PTE_r$ and then discuss how
proper solutions should be defined for the $\PTE_{\ge 2}$ from the combinatorial perspective.
Next in \Cref{subsect:design}, we briefly introduce various classes of combinatorial designs such as orthogonal arrays (OAs) and block $t$-designs, each of which is applicable for `directly' constructing high-dimensional PTE solutions in \Cref{sect:design}; in \Cref{sect:lifting}, we present dimension-lifting constructions by combining lower-dimensional PTE solutions and OAs.

%%%%%%%%%%%%%%%%%%%%%%%%%%%%%%%%%%%%%%%%%%%
%%%%%%%%%%%%%%%%%%%%%%%%%%%%%%%%%%%%%%%%%%%
\subsection{The $r$-dimensional PTE problem and combinatorially proper solutions}\label{sect:PTEr}

We start with the precise definition of the $\PTE_r$.

\begin{problem} [{\cite{Alpers-Tijdeman}}]  \label{PTEr}
The {\it $r$-dimensional PTE problem $(\PTE_r)$ of degree $m$ and size $n$} asks whether there exists a disjoint pair of multisets
\[
A = \{(a_{11}, \ldots, a_{1r}), \ldots, (a_{n1}, \ldots, a_{nr})\}, \quad B= \{(b_{11}, \ldots, b_{1r}), \ldots, (b_{n1}, \ldots, b_{nr})\} \subset \mathbb{Z}^r
\]
such that
\[
\sum_{i=1}^n \prod_{j=1}^r a_{ij}^{k_j} =\sum_{i=1}^n \prod_{j=1}^r b_{ij}^{k_j},
\]
where the summations are taken over all combinations of nonnegative integers $k_1, \ldots, k_r$ satisfying the condition $1 \leq k_1+ \cdots +k_r \leq m$.
We often use the notation $[A]=^n_m [B]$ or $[\mathbf{a}_1,\ldots,\mathbf{a}_n] =_m^n [\mathbf{b}_1,\ldots,\mathbf{b}_n]$ for the solution $(A,B)$.
\end{problem}

\begin{remark}
\begin{enumerate}
\item[(i)] We can also define the $\PTE_r$ over $\mathbb{Q}$.
Note that an integer solution of the $\PTE_r$ can be obtained from a rational solution by clearing denominators.
In the following, we consider rational solutions unless otherwise stated.
\item[(ii)]
Given $\mathbf{a}_1, \ldots, \mathbf{a}_n, \mathbf{b}_1, \ldots, \mathbf{b}_{n-1} \in \mathbb{Q}^r$, we can readily obtain a $\PTE_r$ solution of degree $1$ as $A=\{\mathbf{a}_i \}$ and $B=\{\mathbf{b}_i \}$ with $\mathbf{b}_n:=\sum_{i=1}^n \mathbf{a}_i-\sum_{i=1}^{n-1} \mathbf{b}_i$.
Therefore, in this paper, we are mainly concerned with the case where $m \geq 2$. 
\end{enumerate}
\end{remark}

Most of the solutions appearing in this paper have the following two properties:

\begin{definition} \label{def:symmetric}
\begin{enumerate}[label=$(\roman*)$] 
\item (Cf.\ \cite[p.\ 86]{Borwein2002}) A solution $[\mathbf{a}_1, \ldots, \mathbf{a}_n] =_m^n [\mathbf{b}_1, \ldots, \mathbf{b}_n]$ is {\it symmetric} if
\begin{align*}
\{\mathbf{a}_i  \mid 1 \leq i \leq n \}=\{-\mathbf{a}_i  \mid 1 \leq i \leq n\},
&& \{\mathbf{b}_i  \mid 1 \leq i \leq n \}=\{-\mathbf{b}_i  \mid 1 \leq i \leq n\}
\end{align*}
as multisets.
\item (\cite[p.\ 26]{MLSU})
A solution $[\mathbf{a}_1, \ldots, \mathbf{a}_n] =_m^n [\mathbf{b}_1, \ldots, \mathbf{b}_n]$ is {\it linear} if there exists $S \subset \{1,\ldots,n\}$ such that $\sum_{i \in S} \mathbf{a}_i = \sum_{i \in S} \mathbf{b}_i = \boldsymbol{0}$.
\end{enumerate}
\end{definition}

\begin{remark}[{\cite[Definition 2.4]{PIL2023}}]
Prugsapitak et al.\ \cite{PIL2023} use the term {\it reduced solution}, instead of linear solution for the matrix ring analog of the classical PTE problem.
\end{remark}

The following is a two-dimensional extension of the Borwein solution (\ref{Borwein}):

\begin{theorem} [{\cite[Proposition 5.12]{MS2025}, Two-dimensional Borwein solution}] \label{2DBorwein} 
There exists an ideal solution of the $\PTE_2$ given by
\begin{equation*}
[\pm \mathbf{a}_1, \pm \mathbf{a}_2, \pm \mathbf{a}_3]=_5^6 [\pm \mathbf{b}_1, \pm \mathbf{b}_2, \pm \mathbf{b}_3],
\end{equation*}
where
\small
\begin{align*}
\mathbf{a}_1:=(2(b+a),-ab-b-a+3), && \mathbf{a}_2:= (-ab-b-a+3, ab-b-a-3), \\
\mathbf{a}_3 := (ab-b-a-3, 2(b+a)), && \mathbf{b}_1:=(2(b-a),ab-b+a+3),  \\
\mathbf{b}_2:=(ab-b+a+3,-ab-b+a-3) , && \mathbf{b}_3 := (-ab-b+a-3, 2(b-a)).
\end{align*}
\normalsize
\end{theorem}

\begin{remark}
\label{rem:linear0}
The two-dimensional Borwein solution is clearly symmetric and also has the linearity condition that
\begin{equation} \label{eq:linear0}
\mathbf{a}_1+\mathbf{a}_2+\mathbf{a}_3= \mathbf{b}_1+\mathbf{b}_2+\mathbf{b}_3=\mathbf{0}.
\end{equation}
\end{remark}

The following fact is simple but useful for further arguments below:

\begin{proposition} [{\cite[Proposition 4.1.3]{Ghiglione}}] \label{ts}
Assume $[A]=_m^n [B]$ for the $\PTE_r$.
Then it holds that for any $M \in \GL_r(\mathbb{Q})$,
\[
[\{ \mathbf{a}M \mid \mathbf{a} \in A\}]=_m^n [\{ \mathbf{b}M \mid \mathbf{b} \in B\}].
\]
\end{proposition} 

The following construction method for the $\PTE_2$, developed independently by Lorentz~\cite[Section 4.3]{Lorentz1949} and Alpers-Tijdeman~\cite[Theorem 12]{Alpers-Tijdeman}, gives a strong motivation for the present paper.

\begin{theorem}[{Lorentz--Alpers--Tijdeman (LAT) construction}] \label{thm:LAT0}
Let $(U_1,V_1)$ be a disjoint pair of two-element subsets of $\mathbb{Q}^2$ of type
\[
U_1 := \{ (0,0), (\phi_1+\phi_2, \psi_1+\psi_2)\}, \; V_1 := \{ (\phi_1,\psi_1), (\phi_2, \psi_2)\} \subset \mathbb{Q}^2.
\]
For $i \ge 1$, we inductively define
\[
U_{i+1}:=V_i \cup (\theta_{i+1} (\phi_{i+1},\psi_{i+1})+U_i), \; V_{i+1}:=U_i \cup (\theta_{i+1} (\phi_{i+1},\psi_{i+1})+V_i) \subset \mathbb{Q}^2,
\]
where $\theta_i$ is a sufficiently large rational number such that
\[
V_i \cap (\theta_{i+1} (\phi_{i+1},\psi_{i+1})+U_i)= U_i \cap (\theta_{i+1} (\phi_{i+1},\psi_{i+1})+V_i)=\emptyset.
\]
Then $(U_k,V_k)$ is a $\PTE_2$ solution of degree $k$ and size $2^k$ for any positive integer $k$.
\end{theorem}

A purely-combinatorial interpretation of the LAT construction is provided in Example~\ref{2col} later.

Now, Alpers and Tijdeman \cite[p.\ 404]{Alpers-Tijdeman} define a ^^ trivial solution' 
as
\[
A = \{ (a_1, \ldots, a_1), \ldots, (a_n, \ldots, a_n) \}, \quad B = \{ (b_1, \ldots, b_1), \ldots, (b_n, \ldots, b_n) \}.
\]
According to this definition, the solution
\[
[(0,0,0), (2,1,0), (1,2,0)] =_2^3 [(1,0,0), (2,2,0), (0,1,0)]
\]
(cf.\ \cite[p.\ 405]{Alpers-Tijdeman}) is
not trivial, for which Alpers and Tijdeman use the term {\it proper solution}.
However, this still seems to be a bit strange since the third components of vectors do not make sense.
Moreover, most of the $\PTE_2$ solutions $(A,B)$, given in \cite{Alpers-Tijdeman,Ghiglione,MS2025}, naturally satisfy the {\it full rank condition}, namely $
\rank A=\rank B=r$; a detailed explanation of $\rank A$ and $\rank B$ is given in \cref{proper} below.
The solutions constructed by Ghiglione \cite[Theorems 4.3.17 and 4.3.20]{Ghiglione} also satisfy the full rank condition.

Thus in this paper, we propose a reconsideration of the definition of
the triviality or properness of solutions.
Below we write $M^{\top}$ for the transpose of a matrix $M$.
 
\begin{definition}
[{(Combinatorially) proper solution}] \label{proper}
Let $A = \{\mathbf{a}_1, \ldots,\mathbf{a}_n\}, B = \{\mathbf{b}_1,\ldots,\mathbf{b}_n\}$ be multisets of $r$-dimensional rational row vectors, and suppose that $[A] =_m^n [B]$.
Then the solution $(A,B)$ is {\it (combinatorially) proper} if $\rank A = \rank B = r$, where $A$ and $B$ are regarded as the $n \times r$ matrices whose rows are $(a_{i1},\ldots,a_{ir})$ and $(b_{i1},\ldots,b_{ir})$, respectively.
\end{definition}

As will be realized in the present paper, such proper solutions arise naturally in combinatorics, especially in combinatorial design theory and graph theory, which is the reason for the term ^^ {\it combinatorially proper}'.
%%--- 以下ではPSでCPSを意味することに言及．  memo1 by SAWA20260318
From now on, we will refer to ^^ combinatorially proper solutions' simply as ^^ proper solutions'.
%%---

Here are two typical examples that appear several times throughout this paper.

\begin{example} [{Halving the trivial OA}] \label{ex:halving}
The following is a proper solution of the $\PTE_3$:
\[ 
 [(0,0,0), (0,1,1), (1,0,1), (1,1,0)] =_2^4 [(0,0,1), (0,1,0), (1,0,0), (1,1,1)].
\]
This example comes from a disjoint pair of orthogonal arrays (OAs) of strength $2$ (see \cref{OA}), which is further equivalent to a two-coloring of the binary cube 
$\{0,1\}^3$.
Just for reader's information, essentially the same example can also be found in the context of  {\it separating invariants} in the invariant ring theory (see \cite[p.\ 506, Theorem 2]{LR2021}).
\end{example}

\begin{example}
[{Doubling the Fano plane}] \label{ex:2731}
Let $\mathfrak{T}_7 := \langle (1,\ldots, 7) \rangle$ be a cyclic subgroup of the symmetric group $\mathfrak{S}_7$.
We denote by $\Orb_{\mathfrak{T}_7}(\mathbf{x})$ the $\mathfrak{T}_7$-orbit of a vector $\mathbf{x} \in \mathbb{Q}^7$.
We define
\begin{align*} 
A =\Orb_{\mathfrak{T}_7}((1,1,0,1,0,0,0)), \quad B =\Orb_{\mathfrak{T}_7}((0,0,1,0,1,1,0)).
\end{align*}
These two orbits come from the {\it Fano plane}, a finite projective plane with $7$ points and $7$ lines (see \cref{P2}), which yield a proper $\PTE_7$ solution of degree $2$ and size $7$; see \cref{prop:Hadamard} for more details.
\end{example}

The following gives a fundamental inequality for the size of proper $\PTE_r$ solutions of even degree. 
The matrices $N_A$ and $N_B$, given in Theorem~\ref{thm:NewBound} below, are inspired by the {\it $t$-th incidence matrix} of $t$-designs; see \cite[Theorem 19.8]{LintWilson2001} or the proof of \cref{cor:NewBound2} below.
Substantially the same idea can also be found in the process of deriving Rao's inequality for orthogonal arrays (cf.\ \cite[Theorem 2.1]{HSS1999}) or in a geometric approach to discrete tomography (cf.\ \cite[Definition 4.3.3]{Ghiglione}).

\begin{theorem} [{Combinatorial PTE inequality}] \label{thm:NewBound}
Let $m \geq 2 $.
Let $\Omega$ be a subset of $\mathbb{Q}^r$, and let $A, B \subset \Omega$.
Suppose that $(A,B) = (\{ \mathbf{a}_i \}_{i=1}^n, \{ \mathbf{b}_i \}_{i=1}^n)$ is a proper $\PTE_r$ solution of degree $2t$ and size $n$.
Let $\{\mathbf{t}_i \}$ be a basis of the space $\mathcal{P}_t(\Omega)$ of all polynomials of degree $\leq t$ defined on $\Omega$.
We define two matrices $N_A$ and $N_B$ of size $\dim_{\mathbb{Q}} \mathcal{P}_t(\Omega) \times n$ as follows:
\begin{align} \label{NAB}
N_A = (\mathbf{t}_i (\mathbf{a}_j) ), && N_B = (\mathbf{t}_i (\mathbf{b}_j)). 
\end{align}
Moreover, we suppose that $\rank {[N_A \; N_B]}=\dim_{\mathbb{Q}} \mathcal{P}_t(\Omega)$.
Then it holds that
\begin{equation} \label{eq:bound2}
n \ge \dim_{\mathbb{Q}} \mathcal{P}_t(\Omega).
\end{equation}
\end{theorem}

\begin{remark} \label{rem:binary0}
Inequality (\ref{eq:bound2}) depends on the choice of the domain $\Omega$ on which polynomials are defined.
In algebraic combinatorics, particularly in design theory and related areas, $\mathfrak{S}_r$-invariant finite domains such as the binary hypercube $C^r := \{0,1\}^r$ or the binary sphere $S_k^{r-1} := \{ (x_1,\ldots,x_r) \in \{0,1\}^r \mid \sum_{i=1}^r x_i^2=k \}$, are often candidates for $\Omega$.
This paper also focuses on such binary domains, especially when dealing with minimal examples with respect to our inequality.
The nonbinary case is also of theoretical interest, but a detailed discussion will be left as a subsequent work.
\end{remark}

The following corollary to \cref{thm:NewBound} for the binary spheres is used for the arguments in Subsection~\ref{subsect:Hadamard}, where we write $D^{\{k\}}$ for the set of all $k$-element subsets of a finite set $D$.

\begin{corollary} \label{cor:NewBound2}
Let $t,k,r$ be positive integers such that $t \le k \le r-t$.
Let $\Omega = S_k^{r-1}$ and $A, B \subset \Omega$.
Suppose that $(A,B)$ is a proper $\PTE_r$ solution of degree $2t$ and size $n$.
With the notation $N_A,N_B$ of \cref{thm:NewBound}, we moreover suppose that  $\rank {[N_A \; N_B]} = \dim_{\mathbb{Q}} \mathcal{P}_t(\Omega)$. 
Then it holds that
\begin{equation} \label{eq:binary1}
n \ge \binom{r}{t}. 
\end{equation}
\end{corollary}
\begin{proof} 
Let $R$ be an $r$-element set, and consider the incidence structure $\mathcal{D} = (R, R^{\{k\}})$.
The $t$-th incidence matrix $N_t$ of $\mathcal{D}$ is defined to be a $(0,1)$-matrix with rows (resp.\ columns) indexed by all $t$-element (resp.\ $k$-element) subsets of $R$, and with entry $1$ in row $T \in R^{\{t\}}$ and column $K \in R^{\{k\}}$ if $T \subseteq K$, $0$ otherwise (cf.~\cite[Theorem 19.8]{LintWilson2001}).
It is well known (cf.~\cite[Theorem 19.8]{LintWilson2001},~\cite[p.\ 252, Remark]{Wilson1984}) that
$\rank N_t = \binom{r}{t} $ if $t \le k \le r-t$.
Note that $\dim_{\mathbb{Q}} \mathcal{P}_t(S_k^{r-1}) = \rank N_t$, since any polynomial of degree $<t$ is written as a homogeneous polynomial of degree exactly $t$ on $S_k^{r-1}$.
By the assumption, we have $\dim_{\mathbb{Q}} \mathcal{P}_t(\Omega)=\rank [N_A \; N_B]$.
The corollary thus follows by \cref{thm:NewBound}.
\end{proof}

In \Cref{sect:minimal}, we give a proof of \cref{thm:NewBound} together with some minimal examples with respect to our bound (\ref{eq:binary1}).

%%%%%%%%%%%%%%%%%%%%%%%%%%%%%%%%%%%%%%%%%%%%%%
%%%%%%%%%%%%%%%%%%%%%%%%%%%%%%%%%%%%%%%%%%%%%%
\subsection{Combinatorial arrays and block designs} \label{subsect:design} 

In this subsection, we introduce two important classes of combinatorial designs, namely 
combinatorial arrays such as orthogonal arrays (OA), block designs such as group divisible designs (GDD), each playing significant roles hereafter. 
Some of the information, given below without detailed explanations, can be found in \cite{HSS1999,LintWilson2001}.

%%%%%%%%%%%%%%%%%%%%%%%%%%%%%%%%%%%%%%%%%%%%%%%%%%%%%
%%%%%%%%%%%%%%%%%%%%%%%%%%%%%%%%%%%%%%%%%%%%%%%%%%%%%
\subsubsection{Combinatorial array} \label{subsubsect:OA}

\begin{definition} [{\cite[p.\ 1]{HSS1999} Orthogonal array (OA)}] \label{def:OA}
Let $s$, $t$, $\lambda$, $r$, $\ell$ be positive integers.
An {\it orthogonal array (OA) of strength $t$, index $\lambda$ and $s$ levels} is an $\ell \times r$ array (matrix) with $s$ symbols such that the rows of any $\ell \times t$ subarray consist of $\lambda$ copies of all $s^t$ ordered $t$-tuples of $s$ symbols.
Such an array is denoted by $\OA(\ell,r,s,t)_{\lambda}$.
\end{definition}

\begin{example}
[Halving the trivial OA, {\cref{ex:halving} continued}] \label{OA}
Two $4 \times 3$ arrays given by
\begin{align*}
{\begin{bmatrix}
0& 0 & 1 & 1\\
0& 1 & 0 & 1\\
0& 1 & 1 & 0
\end{bmatrix}}^{\top}, \quad
{\begin{bmatrix}
0& 0 & 1 & 1\\
0& 1 & 0 & 1\\
1& 0 & 0 & 1
\end{bmatrix}}^{\top}
\end{align*}
are both $\OA(4,3,2,2)_1$.
Note that if one regards $\{0,1\}$ as the finite field $\mathbb{F}_2$ of order $2$, then the left one is a two-dimensional subspace of $\mathbb{F}_2^3$.
Such OAs are often called {\it linear}~\cite[Definition 3.4]{HSS1999}.
\end{example}

It is remarked that two OAs given in \cref{OA} have no rows in common. 
To explain more general situations, we introduce the notion of disjoint OAs.

\begin{definition} [{Disjoint OA, \cite[pp.\ 243--244]{CD2007}}] \label{def:disjointOA}
A family of $\alpha$ $\OA(\ell,r,s,t)_\lambda$ is called {\it disjoint} if any two of them do not have a row in common.
In particular when $\alpha = s^r/\ell$, this is called a {\it large set}.
\end{definition}

\begin{remark} \label{rem:resilient}
Let $r,m,t$ be positive integers such that $r \geq m+t$.
A function 
\[
f: \mathbb{F}_2^r \rightarrow\mathbb{F}_2^m
\]
is {\it $t$-resilient} if the inverse image $f^{-1}(y_1,\ldots,y_m)$ forms the rows of an $\OA(2^{r-m},r,2,t)_{2^{r-m-t}}$ for all $(y_1,\ldots,y_m) \in \mathbb{F}_2^m$ (see {\cite[pp.\ 243--244]{HSS1999}}).
A $t$-resilient function is equivalent to a large set of $\OA(2^{r-m}, r, 2,t)_{2^{r-m-t}}$  (\cite[p.\ 244]{HSS1999}). 
Given a linear $\OA(2^{r-m}, r, 2, t)_{2^{r-m-t}}$, one can easily obtain a large set by taking the $2^m$ disjoint copies (translates) of the original.
The reader can find a huge number of linear OAs in \cite{HSS1999}.
\end{remark}

The following gives a graph-theoretic interpretation of the LAT construction (\cref{thm:LAT0}):

\begin{example}
[\cref{OA}, continued]
\label{2col}
The binary hypercube $C^r = \{0,1\}^r$ is a trivial $\OA(2^r,r,2,r)_1$.
Let $A$ (resp.\ $B$) be the set of all $(0,1)$-vectors with an even (resp.\ odd) number of $1$'s.
Then it is not entirely obvious but shown that $A$ and $B$ are $\OA(2^{r-1}, r, 2, r-1)_1$, respectively. 
We remark that a partition of $C^r$ into $A$ and $B$ is equivalent to a two-coloring of the hypercube $C^r$. 
The idea of constructing $\PTE_{\ge 2}$ solutions via two-colorings of the hypercube provides a graph-theoretic understanding of the LAT construction.
\end{example}

The following is a variation of the concept of OA.

\begin{definition}
[{\cite[Definition 6.41]{HSS1999} Type-I OA}] \label{def:OAII}
Let $s$, $t$, $\lambda$, $r$, $\ell$ be positive integers such that $s \ge t$.
An {\it orthogonal array of Type I} of strength $t$, index $\lambda$ and $s$ levels is an $\ell \times r$ array (matrix) with $s$ symbols such that the rows of any $\ell \times t$ subarray consist of $\lambda$ copies of all $s!/(s-t)!$ permutations of $t$ distinct symbols. Such an array is denoted by $\OA_I(\ell,r,s,t)_{\lambda}$.
\end{definition}

\begin{example} \label{ex:OAII}
Two matrices defined by
\begin{align*} {\begin{bmatrix}
0&1&2\\
1&2&0
\end{bmatrix}}^{\top},
\qquad
{\begin{bmatrix}
0&1&2&1&2&0\\
1&2&0&0&1&2\\
2&0&1&2&0&1
\end{bmatrix}}^{\top}
\end{align*}
are $\OA_I(3,2,3,1)_1$ and $\OA_I(6,3,3,3)_1$, respectively.
\end{example}

We now look at a different type of combinatorial array from OA.

\begin{definition}
A {\it Latin square of order $\ell$} is an $\ell \times \ell$ array with $\ell$ symbols such that each symbol appears exactly once in each row and column, respectively.
\end{definition}

\begin{example} \label{ex:LS}
Two matrices defined by
\[
L_1=
\begin{bmatrix}
1 & 2\\
2 & 1
\end{bmatrix}, \quad
L_2=
\begin{bmatrix}
1 & 3 & 2\\
2 & 1 & 3\\
3 & 2 & 1
\end{bmatrix}
\]
are Latin squares of order $2$ and $3$, respectively.
\end{example}

Combinatorial arrays given above can be applied to produce $\PTE_r$ solutions in Subsections~\ref{subsect:CA},~\ref{sect:array_lifting} and~\ref{sect:cp_lift}.

%%%%%%%%%%%%%%%%%%%%%%%%%%%%%%%%%%%%%%%%
%%%%%%%%%%%%%%%%%%%%%%%%%%%%%%%%%%%%%%%%
\subsubsection{Block design} \label{subsubsect:BD}
Throughout this subsection, we use the notation $D^{\{k\}}$
as defined just before \cref{cor:NewBound2}.

\begin{definition} [{Group divisible $t$-design (GDD)}] \label{def:GDD}
Let $\lambda, t, k, g, v, r$ be positive integers such that $t \le k \le g \le r = gv$.
Let $R$ be a finite set of $r$ elements ({\it points}), $\mathcal{B} \subset R^{\{k\}}$ ({\it blocks}), and $\mathcal{G} = \{G_1,\ldots,G_g\} \subset R^{\{v\}}$ ({\it groups}). 
A {\it group divisible design (GDD) of group type $v^g$, strength $t$, block size $k$ and index $\lambda$}, is a triple $(R,\mathcal{B},\mathcal{G})$ which satisfies the following three conditions:
\begin{enumerate}[label=$(\roman*)$]
\item $R=\bigsqcup_{i=1}^g G_i$. 
\item $|B \cap G_i| \leq 1$ for all $B \in \mathcal{B}$ and $G_i \in \mathcal{G}$.
\item $|\{B \in \mathcal{B} \mid T \subset B \}|=\lambda$ for all $T \in R^{\{t\}}$ such that $|T \cap G_i| \leq 1$ for all $i$.
\end{enumerate}
Such a system is denoted by $\GDD_{\lambda}(t,k,r)$ of type $v^g$.
A GDD for $v=1$ is called a {\it (block) $t$-design} and denoted by $t$-$(r,k,\lambda)$. 
\end{definition}

\begin{remark} \label{rem:GDD1}
\cref{def:GDD} is a specialization of the more general definition of GDD adopted in many previous works on combinatorial designs (see e.g.~\cite[p.159]{R1990}).
Such a general treatment of GDD would lead to technical complexity and obscure one of the main aims of this paper, namely, creating the relationship between GDD and $\PTE_r$. This is the reason why \cref{def:GDD} is adopted for the definition of GDD in this paper.
\end{remark}

\begin{example}[{Affine plane}]
Let $R=\{1,2, \ldots, 9\}$, $\mathcal{G} = \{ \{ 1,2,3\}, \{4,5,6\}, \{7,8,9 \} \}$ and
\[
\mathcal{B} = \{ \{1,4,7 \}, \{1,5,8\}, \{1,6,9 \}, \{2,6,8 \}, \{2,4,9\}, \{2,5,7 \}, \{3,5,9 \}, \{3,6,7\}, \{3,4,8 \} \}.
\]
Then $(R,\mathcal{B},\mathcal{G})$ is a $\GDD_1(2,3,9)$ of type $3^3$.
A classical fact in combinatorial design theory (cf.~\cite[Chapter 21]{LintWilson2001}) states that a $\GDD_1(2,q,q^2)$ of type $q^q$ is equivalent to a $2$-$(q^2,q,1)$ design, or the \emph{affine plane of order $q$}.
\end{example}

As in the case of OAs, we introduce the disjointness for GDDs.

\begin{definition}[Disjoint GDD] \label{def:disjointGDD}
A family of $\alpha$ GDDs with the same
parameters,
%%% parematers,
%% parameters
%% より正しい say の使い方   memo2 by SAWA20260318
say $(R,\mathcal{B}_1,\mathcal{G}), \ldots, (R,\mathcal{B}_\alpha,\mathcal{G})$, is called {\it disjoint} if $\mathcal{B}_i \cap \mathcal{B}_j = \emptyset$ for all $i \neq j$.
\end{definition}

\begin{example}[{Doubling the Fano plane, \cref{ex:2731} continued}] \label{P2} 
Let $R = \mathbb{F}_7$, and define $\mathcal{A}, \mathcal{B} \subset R^{\{3\}}$ by
\begin{align*}
\mathcal{A}:=\{ \{i, i+1, i+3 \} \mid i \in R \}, \quad \mathcal{B}:=\{ \{i, i+2, i+3 \} \mid i \in R \}. 
\end{align*}
Then we have a disjoint pair of $2$-$(7,3,1)$ designs $(R,\mathcal{A})$ and $(R,\mathcal{B})$, each providing a combinatorial technique to construct the Fano plane $\PG(2,2)$. 
As seen in \cref{ex:2731}, the block sets $\mathcal{A}$ and $\mathcal{B}$ also have expressions in terms of $(0,1)$-vectors as
$A = {\rm Orb}_{\mathfrak{C}_7}((1,1,0,1,0,0,0))$ and $B = {\rm Orb}_{\mathfrak{C}_7}((0,0,1,0,1,1,0))$, respectively.
\end{example}

In Subsection~\ref{subsect:BD} we present combinatorial constructions of $\PTE_r$ solutions via block designs.

We now look at one more combinatorial concept related to $2$-designs.

\begin{definition}[{Hadamard matrix}]
An $h \times h$ matrix with entries $\pm 1$, say $H$, is called an {\it Hadamard matrix of order $h$} if
\[
H H^{\top} = h I_h,
\]
where $I_h$ is the identity matrix of order $h$.
Any Hadamard matrix can be {\it normalized} such that all entries of the first row and the first column are one.
\end{definition}

It is not entirely obvious but shown (cf.~\cite[Theorem 18.1]{LintWilson2001}) that if $h \ge 3$ and an Hadamard matrix of order $h$ exists, then $h \equiv 0 \pmod{4}$.
A classical fact concerning Hadamard matrices is the equivalence between Hadamard matrices of order $4k$ and $2$-$(4k-1,2k-1,k-1)$ designs (cf.\ ~\cite[Example 19.3]{LintWilson2001}).
A $2$-$(4k-1,2k-1,k-1)$ design is often called an {\it Hadamard $2$-design}.

We now treat one of the most famous constructions of Hadamard $2$-designs.
Let $p$ be a prime number congruent to $3$ modulo $4$.
We define $Q_p$ to be a $p\times p$ matrix whose $(i,j)$-entry is given by the Legendre symbol $(\frac{i-j}{p})$.
Then the matrix defined by
\begin{align} \label{Pp+1}
P_{p+1}:=\begin{bmatrix}
0 & 1 & \cdots & 1\\
1 & & &\\
\vdots& &Q_p  &\\
1& & & 
\end{bmatrix}
\end{align}
is called a {\it Paley matrix} (\cite[pp.\ 202--203]{LintWilson2001}).
The following lemma is used to construct solutions of degree $2$ that are minimal with respect to our bound (\ref{eq:binary1}); see \cref{thm:tight2} for details. 

\begin{lemma} [{Cf.\ \cite[Theorem 18.5 and Example 19.3]{LintWilson2001}}] \label{Paley}
Let $p$ be a prime number such that $p \equiv 3 \pmod{4}$ and $p \ge 7$.
Let $J_p$ be the all-one matrix of order $p$.
Then the following hold:
\begin{enumerate}[label=$(\roman*)$]
\item $Q_p$ is a circulant matrix, and so is $(Q_{p}-I_{p}+J_{p})/2$.
\item $I_{p+1}+P_{p+1}$ is a normalized Hadamard matrix of order $p+1$, or equivalently, $(Q_p - I_p + J_p)/2$ is the incidence matrix of an Hadamard $2$-$(p, (p-1)/2, (p-3)/4)$ design.
\end{enumerate}
\end{lemma}

%%%%%%%%%%%%%%%%%%%%%%%%%%%%%%%%%%%%%%%%%%%%
%%%%%%%%%%%%%%%%%%%%%%%%%%%%%%%%%%%%%%%%%%%%
\section{Combinatorial PTE inequality and tight solutions} \label{sect:minimal}

In this section we first give a proof of the combinatorial PTE inequality (\cref{thm:NewBound}) and then show some examples of tight solutions with respect to our inequality.

%%%%%%%%%%%%%%%%%%%%%%%%%%%%%%%%%%%%%%%%%%%%
%%%%%%%%%%%%%%%%%%%%%%%%%%%%%%%%%%%%%%%%%%%%
\subsection{Proof of the combinatorial PTE inequality} \label{subsect:proof}

The following lemma is key in the proof of \cref{thm:NewBound}, where a multiset of $r$-dimensional row vectors is identified with a matrix with $r$ columns:

\begin{lemma} \label{SAB}
Let $m \ge 2$.
Suppose that $A$ (resp. $B$) is a multiset of $n$ $r$-dimensional rational row vectors $\mathbf{a}_i$ (resp. $\mathbf{b}_i$) for which $[A]=_m^n [B]$ holds.
Let $[A^{\top} \; B^{\top}]$ be the matrix defined by
\[
 [\mathbf{a}_1^{\top} \; \cdots \; \mathbf{a}_n^{\top} \; \mathbf{b}_1^{\top} \; \cdots \; \mathbf{b}_n^{\top}].
\]
Then it holds that
\[
\rank {[A^{\top} \; B^{\top}]} = \rank A = \rank B.
\]
\end{lemma}

\begin{proof}
Suppose that
$\rank {[A^{\top} \; B^{\top}]}=l$. 
By suitably performing elementary row operators, the $r \times 2n$ matrix $[A^{\top} \; B^{\top}]$ can be transformed as
\[
 [\tilde{A}^{\top} \; \tilde{B}^{\top}] := [ \tilde{\mathbf{a}}_1^{\top} \; \cdots \; \tilde{\mathbf{a}}_n^{\top} \; \tilde{\mathbf{b}}_1^{\top} \; \cdots \; \tilde{\mathbf{b}}_n^{\top}]
\]
such that any distinct rows are mutually orthogonal.
Hence it holds that for any distinct $1\le k,k' \le r$,
\[
0 = \sum_{i=1}^n \tilde{a}_{ik} \tilde{a}_{ik'} + \sum_{i=1}^n \tilde{b}_{ik} \tilde{b}_{ik'},
\]
where $\tilde{\mathbf{a}}_i = (\tilde{a}_{i1},\ldots,\tilde{a}_{ir}),\;\tilde{\mathbf{b}}_i = (\tilde{b}_{i1},\ldots,\tilde{b}_{ir})$ for each $i = 1\ldots,n$.
Since $[\{ \tilde{\mathbf{a}}_i \}_{i=1}^n ] =_m^n [ \{  \tilde{\mathbf{b}}_i \}_{i=1}^n ]$ by \cref{ts}, we have
\[
 \sum_{i=1}^n \tilde{a}_{ik} \tilde{a}_{ik'}=\sum_{i=1}^n \tilde{b}_{ik} \tilde{b}_{ik'},
\]
and hence
\[
0 = 2 \sum_{i=1}^n \tilde{a}_{ik} \tilde{a}_{ik'}.
\]
Since
$\rank {[\tilde{A}^{\top} \; \tilde{B}^{\top}]}=\rank {[A^{\top} \; B^{\top}]} = l$ and the first $l$ rows of $[\tilde{A}^{\top} \; \tilde{B}^{\top}]$ are nonzero, we have
\[
 \sum_{i=1}^n \tilde{a}_{ik}^2 + \sum_{i=1}^n \tilde{b}_{ik}^2 \neq 0 \ \text{ for every $k=1, \ldots, l$}.
\]
Since $[\{ \tilde{\mathbf{a}}_i \}_{i=1}^n ] =_m^n [ \{  \tilde{\mathbf{b}}_i \}_{i=1}^n ]$, we have
\[
 \sum_{i=1}^n \tilde{a}_{ik}^2 =\sum_{i=1}^n \tilde{b}_{ik}^2
\]
and hence
\[
 \sum_{i=1}^n \tilde{a}_{ik}^2 \neq 0 \ \text{ for every $k=1, \ldots, l$}.
\]
Therefore, we have $\rank {[\tilde{\mathbf{a}}_1^{\top} \; \cdots \; \tilde{\mathbf{a}}_n^{\top}]}=l$ and
hence $\rank {[ \mathbf{a}_1^{\top} \; \cdots \; \mathbf{a}_n^{\top} ]}=l$. 
Similarly, we conclude that $\rank {[\mathbf{b}_1^{\top} \; \cdots \; \mathbf{b}_n^{\top} ]}=l$. 
\end{proof}

We now give a proof of \cref{thm:NewBound}.

\begin{proof} [{Proof of \cref{thm:NewBound}}]
Since $[A]=_{2t}^n [B]$ and $\deg (\mathbf{t}_i) \leq \deg (\mathbf{t}_i \mathbf{t}_j) \leq 2t$, 
it holds that for all $1 \leq i \leq j \leq t$,
\begin{align*}
\sum_{k=1}^n \mathbf{t}_i(\mathbf{a}_k) &=\sum_{k=1}^n \mathbf{t}_i(\mathbf{b}_k), \\
\sum_{k=1}^n \mathbf{t}_i(\mathbf{a}_k)\mathbf{t}_j(\mathbf{a}_k) &=\sum_{k=1}^n \mathbf{t}_i(\mathbf{b}_k)\mathbf{t}_j(\mathbf{b}_k).
\end{align*}
With the notation $N_A,N_B$ of (\ref{NAB}), we have $[N_A] =_2^n [N_B]$
for the $\PTE_{\dim_{\mathbb{Q}} \mathcal{P}_t(\Omega)}$.
This implies that $n \geq \dim_{\mathbb{Q}} \mathcal{P}_t(\Omega)$ by the assumption $\rank {[N_A \; N_B]}=\dim_{\mathbb{Q}} \mathcal{P}_t(\Omega)$ and \cref{SAB}.
\end{proof}

\begin{definition} [Tight solution of even degree] \label{def:minimal}
With the notation $N_A,N_B,\mathcal{P}_t(\Omega)$ of \cref{thm:NewBound}, a $\PTE_r$ solution of degree $2t$ and size $n$ is {\it tight} if $n = \rank {[N_A \; N_B]} = \dim_{\mathbb{Q}} \mathcal{P}_t(\Omega)$. 
\end{definition}

\begin{remark}
\label{rem:tight}
\cref{SAB} implies that any tight solution is
proper.
% nontrivial.
\end{remark}

%%%%%%%%%%%%%%%%%%%%%%%%%%%%%%%%%%%%%%%
%%%%%%%%%%%%%%%%%%%%%%%%%%%%%%%%%%%%%%%
\subsection{Tight solutions associated with distinctive $t$-designs and OA}\label{subsect:Hadamard}

In this subsection, we mainly take $\Omega \subset \mathbb{Q}^r$ to be the $r$-dimensional binary sphere $S_k^{r-1}$, as introduced in \cref{rem:binary0}.
The only intent of this subsection is to have a look at some examples of tight solutions derived from distinctive $t$-designs or OA related to the LAT construction (\cref{thm:LAT0}).

We begin with a simple method for generating a number of tight solutions of degree $2$.

\begin{proposition} \label{prop:Hadamard}
Let $k \ge 2$ be an integer.
Suppose that there exists a disjoint pair of Hadamard $2$-$(4k-1,2k-1,k-1)$ designs.
Then there exists a tight solution of degree $2$ for the binary sphere $S_{2k-1}^{4k-2}$.
\end{proposition}

\begin{proof}
By \cref{TdesigntoPTE}, we obtain a $\PTE_{4k-1}$ solution of degree $2$ and size $4k-1$.
It is obvious that our Hadamard design has $4k-1$ ($= \dim \mathcal{P}_1(S_{2k-1}^{4k-2})$) blocks, which is tight according to \cref{cor:NewBound2}.
\end{proof}

\begin{example} [{Doubling the Fano plane, \cref{ex:2731} continued}] 
\label{ex:2731-rev}
The sets $A$ and $B$ given in \cref{ex:2731} yield a $\PTE_7$ solution of degree $2$ and size $7$.
By \cref{prop:Hadamard}, this example is tight for the binary sphere $S_3^6$.
\end{example}

As a generalization of \cref{ex:2731-rev}, we present an infinite family of tight solutions of degree $2$.

\begin{theorem} \label{thm:tight2}
Let $p$ be a prime number such that $p \equiv 3 \pmod{4}$ and $p \ge 7$.
Then there exists a tight solution of degree $2$ for the binary sphere $S_{(p-1)/2}^{p-1}$.
\end{theorem}

\begin{proof}
Let $p \geq 7$ be a prime number congruent to $3$ modulo $4$.
Let $Q_p$ be the matrix defined in (\ref{Pp+1}).
Then by \cref{Paley} (i), $(Q_{p}-I_{p}+J_{p})/2$ is a circulant matrix.
The first column of $(Q_p - I_p + J_p)/2$ is the {\it characteristic vector} (see (\ref{eq:char})) of the set $\mathbb{F}_p^{\times 2}$ of nonzero quadratic residues in $\mathbb{F}_p$.
For each $a \in \mathbb{F}_p$, let
\[
\beta_a := |\mathbb{F}_p^{\times 2} \cap (-\mathbb{F}_p^{\times 2}+a)|.
\]
Since $p \equiv 3 \pmod{4}$, we have
\[
\mathbb{F}_p^{\times 2} \cup (-\mathbb{F}_p^{\times 2}) \cup \{0\} = \mathbb{F}_p.
\]
Hence it follows that
\begin{align*}
\beta_a
&= |\mathbb{F}_p^{\times 2}  \cap (\mathbb{F}_p + a)| - |\mathbb{F}_p^{\times 2}  \cap \{a\}| - |\mathbb{F}_p^{\times 2}  \cap (\mathbb{F}_p^{\times 2} +a)| \\
&= 
\begin{cases}
 0 & (a=0), \\
\frac{p-1}{2} - 1 - \frac{p-3}{4} = \frac{p-3}{4} & (a \in \mathbb{F}_p^{\times 2}), \\
\frac{p-1}{2} - \frac{p-3}{4} = \frac{p+1}{4} & (a \in -\mathbb{F}_p^{\times 2}),
\end{cases}
 \end{align*}
where \cref{Paley} (ii) is used in the computation of $|\mathbb{F}_p^{\times 2} \cap (\mathbb{F}_p^{\times 2}+a)|$.
We thus conclude that
$\beta_a \le (p+1)/4< (p-1)/2$ if $p > 3$.
This implies that two $2$-$(p,(p-1)/2,(p-3)/4)$ designs $(\mathbb{F}_p, \mathcal{B}_i)$, where
\[
 \mathcal{B}_1 = \{\mathbb{F}_p^{\times 2}+a \mid a \in \mathbb{F}_p \}, \quad \mathcal{B}_2 = \{-\mathbb{F}_p^{\times 2}+a \mid a \in \mathbb{F}_p \},
\]
are disjoint. 
The result then follows from \cref{prop:Hadamard}. 
\end{proof}

\begin{remark} \label{rem:HadamardDesign}
Although we have only dealt with the Paley matrix approach, there are many other known methods to construct Hadamard designs (cf.~\cite[\S~18]{LintWilson2001}), and accordingly many ways to construct tight solutions of degree $2$.
\end{remark}

The following is a spectacular tight solution of degree $4$ for the binary sphere $S_7^{22}$.

\begin{example}
[Tight solution derived from Witt system]
\label{ex:tight4}
Let $\mathcal{F}$ be a collection of $7$-subsets of $\mathbb{F}_{23}$ given by
\begin{align*}
\mathcal{F} :=
\{ &\{0, 1, 2, 3, 5, 14, 17\}, \;
\{0, 1, 2, 6, 7, 19, 21\}, \;
\{0, 1, 2, 8, 11, 12, 18\}, \; \{0, 1, 2, 9, 10, 15, 20\},\\
& \{0, 1, 3, 4, 11, 19, 20\}, \;
\{0, 1, 3, 6, 8, 10, 13\}, \; 
\{0, 1, 3, 7, 9, 16, 18\}, \;
\{0, 1, 4, 6, 9, 12, 17\}, \\
&\{0, 1, 4, 10, 14, 18, 21\}, \;
\{0, 1, 5, 9, 11, 13, 21\}, \;
\{0, 1, 5, 10, 12, 16, 19\}\}.
\end{align*}
Then the set $A$ of the characteristic vectors of the $7$-subsets of type $F + a$, where $a \in \mathbb{F}_{23}$ and $F \in \mathcal{F}$, yields a $4$-$(23,7,1)$ design. This is widely known as {\it Witt system}~(\cite{Witt1937_Mathieu,Witt1937_Steiner}) and the only known $4$-design with $r$ points and $\binom{r}{2}$ blocks; see~\cite[Table 2]{Iwasaki1988}.
Let $B$ be the set of the reverse sequences (vectors) of $A$. 
We can easily check that $A$ and $B$ are disjoint.
Hence by \cref{TdesigntoPTE}, we obtain $[A] =^{253}_4 [B]$.
This is a tight solution of degree $4$ for the binary sphere $S_7^{22}$, since the second incidence matrix $N_{2}$ of $4$-$(23,7,1)$ is a submatrix of $N_A$ (resp. $N_B$) with
$\binom{23}{0}+\binom{23}{1}+\binom{23}{2}$ rows and $\binom{23}{2}$ columns and hence
\[
\rank N_A = \rank N_B  =\rank N_2= \dim_{\mathbb{Q}} \mathcal{P}_2(S_7^{22}) = \binom{23}{2}=253;
\]
see the proof of \cref{cor:NewBound2} for details.
\end{example}

The following degree-four solution is of great importance to the present paper, which is tight for the binary hypercube $C^5 = \{0,1\}^5$.

\begin{example}
[Tight solution derived from LAT construction]
\label{ex:tight4-2}
Let $A$ (resp.\ $B$) be the set of all $5$-dimensional $(0,1)$-vectors with an even (resp.\ odd) number of $1$'s. As seen in \cref{2col}, both $A$ and $B$ are $\OA(16, 5, 2, 4)_1$. These OA were first found by Rao~\cite[p.\ 130]{Rao1947} in the context of design of experiments, which are tight for the binary hypercube $C^5$, since
\[
16 = \binom{5}{0}+\binom{5}{1}+\binom{5}{2} = \dim_{\mathbb{Q}} \mathcal{P}_2 (C^5).
\]
\end{example}

\begin{remark}
\label{rem:nonbinary0}
Let
\[
A= \{(4,0),(1,1),(3,2),(5,2),(0,3),(2,4)\}, \quad B= \{(3,0),(5,1),(0,2),(2,2),(4,3),(1,4)\}.
\]
Then $(A,B)$ is a $\PTE_2$ solution of degree $4$ and size $6$, which is tight, for example when $\Omega$ is taken to be the senary square $\{0,1,2,3,4,5\}^2$.
Further discussion on such nonbinary tight examples are beyond the scope of this paper, and will be extensively investigated in a forthcoming paper.
\end{remark}

%%%%%%%%%%%%%%%%%%%%%%%%%%%%%%%%%%%
%%%%%%%%%%%%%%%%%%%%%%%%%%%%%%%%%%%
\section{Direct constructions of $\PTE_r$ solutions and disjoint combinatorial designs} \label{sect:design}

In this section we present criteria to `directly' construct $\PTE_r$ solutions by taking a disjoint pair of combinatorial designs; particularly in the case of OAs, our criterion substantially generalizes the LAT construction (\cref{thm:LAT0}).
The results of this section give a number-theoretic interpretation of the {\it regularity condition} of combinatorial designs.

%%%%%%%%%%%%%%%%%%%%%%%%%%%%%%%%%%
%%%%%%%%%%%%%%%%%%%%%%%%%%%%%%%%%%
\subsection{OA-based direct construction} \label{subsect:CA}

The following result, which is a generalization of the LAT construction, is
a main theorem of this subsection:

\begin{theorem} [{OA-based construction}] \label{OAtoPTE}
Let $s, t \geq 2$.
Let $X,Y$ be the sets of rows of two disjoint $\OA(\ell,r,s,t)_{\lambda}$.
Then $(X,Y)$ is a proper solution of the $\PTE_r$ of degree $t$ and size $\ell$.
\end{theorem}

To prove our main theorem, we use two auxiliary lemmas.

\begin{lemma} \label{fullrank}
Let $M \in M_{n,r}(\mathbb{Q})$. Suppose that there exist $a$, $b \in \mathbb{Q}$ such that $M^{\top}M = a I_r + bJ_r$.
Moreover, suppose that $a \neq 0$ and $a+rb \neq 0$. Then it holds that $\rank M=r$.
\end{lemma}

\begin{proof}
By standard arguments from linear algebra, we see that the eigenvalues of $M^{\top}M$ are $a+rb$ and $a$ with multiplicity $1$ and $r-1$, respectively, which implies the desired result.
\end{proof}

\begin{lemma} [{Regularity of OA}] \label{OAreg}
Let $t$, $t'$ be positive integers such that $t' \leq t$.
Then an $\OA(\ell,r,s,t)_{\lambda}$ is an $\OA(\ell,r,s,t')_{\lambda_{t'}}$ for some $\lambda_{t'}$.
\end{lemma}

\begin{proof}
See \cite[Problem 2, p.\ 4]{HSS1999}.
\end{proof}

We are now ready to prove \cref{OAtoPTE}.

\begin{proof}[{Proof of \cref{OAtoPTE}}]
Let $X$ and $Y$ be our disjoint $\OA(\ell,r,s,t)_{\lambda}$ with $s, t \geq 2$. 
By \cref{OAreg}, it holds that
\[
\sum_{\mathbf{x} \in X} \prod_{j=1}^r x_j^{k_j} = \sum_{\mathbf{y} \in Y}  \prod_{j=1}^r y_j^{k_j} \quad (1 \le k_1+\cdots+k_r \le t).
\]
Since $X \cap Y=\emptyset$, we obtain $[X]=^{\ell}_t [Y]$.

Next we prove the full rankness of our OA.
Let $\gamma_1, \ldots, \gamma_s$ be the symbols of an $\OA(\ell,r,s,t)_{\lambda}$.
By \cref{OAreg}, the $(i,j)$-th component of $X^{\top} X$ is given by
\begin{align*}
\begin{cases}
\frac{\ell}{s^2}\left(\sum_{i=1}^s \gamma_i^2+ 2 \sum_{1 \leq i<j \leq s} \gamma_i \gamma_j \right)=\frac{\ell}{s^2} \left(\sum_{i=1}^s \gamma_i \right)^2 & (i \neq j), \\
\frac{\ell}{s^2} s\sum_{i=1}^s \gamma_i^2=\frac{\ell}{s} \sum_{i=1}^s \gamma_i^2  & (i=j).
\end{cases}
\end{align*}
Therefore, we have
\begin{align*}
X^{\top}X
&= \frac{\ell}{s^2} \left(\left(s\sum_{i=1}^s \gamma_i^2-\left(\sum_{i=1}^s \gamma_i \right)^2 \right) I_r + \left(\sum_{i=1}^s \gamma_i  \right)^2 J_r \right).
\end{align*}
By the Cauchy--Schwarz inequality, we have
\[
s\sum_{i=1}^s \gamma_i^2 -\left(\sum_{i=1}^s \gamma_i \right)^2 \geq 0.
\]
Therefore, 
\[
\left(s\sum_{i=1}^s \gamma_i^2-\left(\sum_{i=1}^s \gamma_i \right)^2+r \left(\sum_{i=1}^s \gamma_i  \right)^2 \right)
\left(s\sum_{i=1}^s \gamma_i^2-\left(\sum_{i=1}^s \gamma_i \right)^2 \right)=0
\]
if and only if
$\gamma_1 = \cdots = \gamma_s$.
Since our OA has level $s \geq 2$, we obtain $\rank X=r$ by \cref{fullrank}.
\end{proof}

The argument used in the second paragraph of the above proof will appear several times in subsequent sections; for example, see \cref{GDDtoPTE,thm:OAL1}.

\begin{remark} \label{rem:OA}
\begin{enumerate}
\item[(i)] Halving the trivial $\OA(2^r,r,2,r)_1$, which is a special case of \cref{OAtoPTE}, provides a design-theoretic interpretation of the LAT construction (\cref{thm:LAT0}).
\item[(ii)] It turns out that a Type-I OA analogue of \cref{OAtoPTE} is also available. 
For example, both $A = \{(0,1), (1,2), (2,0)\}$ and $B = \{(1,0), (2,1), (0,2)\}$ are the rows of a Type-I OA of strength $1$, giving a $\PTE_2$ solution of degree $1$ (and also of degree $2$).
\end{enumerate}
\end{remark}

%%%%%%%%%%%%%%%%%%%%%%%%%%%%%%%%%%%%%%%
%%%%%%%%%%%%%%%%%%%%%%%%%%%%%%%%%%%%%%%
\subsection{Block-design-based direct construction} \label{subsect:BD}

The following is the main theorem of this subsection, where we use the notation $D^{\{k\}}$ as defined just before \cref{cor:NewBound2}:

\begin{theorem} [GDD-based construction] \label{GDDtoPTE}
Let $(R,\mathcal{B}_1,\mathcal{G})$ and $(R,\mathcal{B}_2,\mathcal{G})$ be a disjoint pair of $\GDD_{\lambda}(t,k,r)$ of type $v^g$ with $b$ blocks, where $k < g$.
Let $X_1$ and $X_2$ be the sets of the characteristic vectors of $\mathcal{B}_1$ and $\mathcal{B}_2$, respectively.
Then $(X_1,X_2)$ is a proper solution of the $\PTE_r$ of degree $t$ and size $b$.
\end{theorem}

To prove our main theorem, we use some auxiliary lemmas.

The following seems to be standard for design theorists but not for researchers in other fields. Below, we give a proof based on the {\it double counting} for those researchers.

\begin{lemma} [{Regularity of GDD}] \label{GDDts}
Let $\lambda, t, k, g, v, r$ be positive integers as above.
Let $s$ be a positive integer such that $s \leq t$.
Let $(R,\mathcal{B},\mathcal{G})$ be a $\GDD_{\lambda}(t,k,r)$ of type $v^g$.
Then the number of blocks containing $S \in R^{\{s\}}$ is given by
\begin{align} \label{ls}
\begin{cases} 
\frac{\lambda \binom{g-s}{t-s} v^{t-s}}{\binom{k-s}{t-s}} (=:\lambda_s) &
(|G \cap S| \le 1 \; \text{for every $G \in \mathcal{G}$}),\\
0 & (\text{otherwise}). 
\end{cases}
\end{align}
In particular, a $\GDD_{\lambda}(t,k,r)$ of type $v^g$ is a $\GDD_{\lambda_s}(s,k,r)$ of type $v^g$.
\end{lemma}

\begin{proof} 
Let $S \in R^{\{s\}}$, which is included in a block of our GDD. We count the number of elements in the set
\[
\mathcal{F}_S = \left\{(T,B) \mid T \in R^{\{t\}}, B \in \mathcal{B}, S \subset T \subset B \right\} 
\]
in two ways. 

First, let $\lambda_S$ be the number of blocks containing $S$.
For each of such blocks, say $B$, the number of subsets $T \in R^{\{t\}}$ of $B$ such that $S \subset T$ is $\binom{k-s}{t-s}$, since $B \setminus T \subset B \setminus S$.
Thus, we have
\[
|\mathcal{F}_S|
= \sum_{\substack{B \in \mathcal{B} \\S \subset B}}
\sum_{\substack{T \in R^{\{t\}} \\ S \subset T \subset B}} 1
=\sum_{\substack{B \in \mathcal{B} \\S \subset B}} \binom{k-s}{t-s}
=\lambda_S \binom{k-s}{t-s}.
\]

Second, the number of $T \in R^{\{t\}}$ such that $S \subset T$ and $|G \cap S| \le 1$ for every $G \in \mathcal{G}$ is given by $\binom{g-s}{t-s}v^{t-s}$.
For each such $T$, the number of blocks $B$ containing $S$ such that $T \subset B$ is $\lambda$. 
Thus, we have 
\[
|\mathcal{F}_S|
 = \sum_{\substack{T \in R^{\{t\}} \\ S \subset T\\ 
\forall G \in \mathcal{G},\; |G \cap T| \leq 1}} \sum_{\substack{B \in \mathcal{B} \\
T \subset B}} 1
 = \sum_{\substack{T \in R^{\{t\}} \\ S \subset T \\ 
\forall G \in \mathcal{G},\; |G \cap T| \leq 1}} \lambda
 =\lambda \binom{g-s}{t-s} v^{t-s},
\]
which is independent of the choice of $S$.
\end{proof}

Now, let $(R,\mathcal{B},\mathcal{G})$ be a $\GDD_{\lambda}(t,k,r)$ of type $v^g$, where $R : = \{1,\ldots, r\}$.
Then each block $B \in \mathcal{B}$ is identified with the {\it characteristic vector}
\begin{align} \label{eq:char}
\mathbf{x}^B:=(x_1^B,\ldots,x_r^B),&& x_i^B :=
\begin{cases}
1 & (i \in B),\\
0 & (i \not\in B).
\end{cases}
\end{align}

The following is a restatement of \cref{GDDts} in an algebraic manner:

\begin{lemma} \label{GDDLem}
Let $(R,\mathcal{B},\mathcal{G})$ be a $\GDD_{\lambda}(t,k,r)$ of type $v^g$.
With the notation $\lambda_s$ of (\ref{ls}), it holds that for any $S:=\{i_1,\ldots, i_s \} \in R^{\{s\}}$,
\begin{align*}
\sum_{B \in \mathcal{B} } \prod_{j=1}^s x_{i_j}^B  =
\begin{cases}
\lambda_s & |G \cap S| \le 1 \; \text{for every $G \in \mathcal{G}$}),\\
0 & (\mbox{otherwise}).
\end{cases}
\end{align*}
\end{lemma}

\begin{proof}
We note that
\[
\sum_{B \in \mathcal{B}} \prod_{j=1}^s x_{i_j}^B  = |\{ B \in \mathcal{B} \mid S \subset B\}|.
\]
By \cref{GDDts}, the right hand side is $\lambda_s$ if $|G \cap S| \le 1$ for every $G \in \mathcal{G}$, and $0$ otherwise.
\end{proof}

Now we prove \cref{GDDtoPTE}.

\begin{proof}[{Proof of \cref{GDDtoPTE}}]
Let
$f(\mathbf{z}):=z_1^{k_1} \cdots z_r^{k_r}$
with $1 \leq k_1+\cdots+k_r \leq t$.
Let $S := \{i_1, \ldots, i_s \}$, and suppose that
 \begin{align*}
 k_i & \begin{cases}
 \geq 1 & (i \in S),\\
 =0 & (\mbox{otherwise}).
 \end{cases}
 \end{align*}
Then by \cref{GDDLem}, we obtain
\[
\sum_{B \in \mathcal{B}_1} f(\mathbf{x}^B) =
\sum_{B \in \mathcal{B}_2} f(\mathbf{x}^B).
\]

For the proof of the properness, let $R = \{1,\ldots, r\}$ and $\mathcal{B} = \{B_1,\ldots, B_b\}$ be the points set and block set of a $\GDD_{\lambda}(t,k,r)$ of type $v^g$, respectively, and suppose that $k < g$.
Let 
\[
  X := ((\mathbf{x}^{B_1})^{\top}, \ldots, (\mathbf{x}^{B_b})^{\top})
\]
be the incidence matrix of the GDD, where $\mathbf{x}^B$ is the characteristic vector defined by (\ref{eq:char}).
After a suitable permutation of rows of $X$ (if necessary), we have
\[
XX^{\top} =I_g \otimes (\lambda_1 I_v) + (J_g-I_g) \otimes (\lambda_2 J_v),
\]
where $\lambda_s$ denotes the number of blocks containing $S \in R^{\{s\}}$ (see (\ref{ls})) and $\otimes$ denotes the Kronecker product of matrices.
The eigenvalues of the Kronecker product $M \otimes N$ are given by the products of the eigenvalues of two matrices $M$ and $N$, respectively.
As both $I_g \otimes (\lambda_1 I_v)$ and $(J_g - I_g) \otimes (\lambda_2 J_v)$ are symmetric and commutative, they can be simultaneously diagonalized, and hence the eigenvalues of $XX^{\top}$ are $\lambda_1$, $\lambda_1+(g-1)\lambda_2 v$ and $\lambda_1-\lambda_2 v$ with multiplicities $g(v-1)$, $1$ and $g-1$, respectively; see also~\cite[p.160]{R1990}.
By (\ref{ls}), we have
\begin{align*}
\frac{\lambda_1}{\lambda_2} &=\frac{\lambda \binom{g-1}{t-1} v^{t-1}}{\binom{k-1}{t-1}} \frac{\binom{k-2}{t-2}}{\lambda \binom{g-2}{t-2} v^{t-2}}\\
&=v \frac{(g-1)!}{(t-1)!(g-t)!} \frac{(t-1)!(k-t)!}{(k-1)!}\frac{(k-2)!}{(t-2)!(k-t)!} \frac{(t-2)!(g-t)!}{(g-2)!}\\
&= v \frac{g-1}{k-1} \neq v,
\end{align*}
where the last inequality follows from the assumption $k < g$. 
Hence we have $\lambda_1 - \lambda_2 v \neq 0$.
We also note that $\lambda_1+(g-1)\lambda_2 v \geq \lambda_1>0$. 
Therefore, we conclude $\rank X=r$ by the argument as used in the second paragraph of the proof of \cref{OAtoPTE}.
\end{proof}

The following corollary to \cref{GDDtoPTE} is a generalization of the tight solution given in \cref{ex:2731}.

\begin{corollary} [$t$-design-based construction] \label{TdesigntoPTE}
Let $(R,\mathcal{B}_1)$ and $(R,\mathcal{B}_2)$ be a disjoint pair of $t$-$(r,k,\lambda)$ with $b$ blocks, where $r>k$.
Let $X_1$ and $X_2$ be the sets of the characteristic vectors of $\mathcal{B}_1$ and $\mathcal{B}_2$, respectively.
Then $(X_1,X_2)$ is a proper solution of the $\PTE_r$ of degree $t$ and size $b$.
\end{corollary}

%%%%%%%%%%%%%%%%%%%%%%%%%%%%%%%%%%%%%%%%%%%%
%%%%%%%%%%%%%%%%%%%%%%%%%%%%%%%%%%%%%%%%%%%%
\section{Lifting constructions} \label{sect:lifting}

In this section we embed a low-dimensional PTE solution into OA or Type-I OA in order to obtain a higher-dimensional solution.
The results of this section provide a number-theoretic interpretation of the so-called {\it recursive constructions} in combinatorial design theory (cf.~\cite[Theorem 18.4, Example 19.13]{LintWilson2001}), and generalize the Borwein solution (\ref{Borwein}) and its two-dimensional extension (\cref{2DBorwein}), and a key proposition in Jacroux's work \cite[Lemma 1]{Jacroux1995} on the construction of sets of integers with equal power sums; see \cref{thm:GeneralizedBorwein1,cor:jac} below.

%%%%%%%%%%%%%%%%%%%%%%%%%%%%%%%%%%%%%%%%%%%%%
%%%%%%%%%%%%%%%%%%%%%%%%%%%%%%%%%%%%%%%%%%%%%
\subsection{Combinatorial array lifting and Borwein solution} \label{sect:array_lifting}

We start with an OA-based lifting construction.

\begin{theorem} [OA lifting] \label{thm:OAL1}
Let $m \ge 2$ be an even integer.
Assume that there exists an $\OA(\ell,r,s,r)_{\lambda}$ with $s \ge m+1$.
Moreover, assume that $\pm A_i, \pm B_i$ are mutually distinct rationals for which
\begin{gather}
   [A_1, \ldots, A_s] =_m^s [B_1, \ldots, B_s], \label{eq:degt0} \\
   \sum_{i=1}^s A_i^{m+2}= \sum_{i=1}^s B_i^{m+2}, \label{eq:degtplus2-0} \\
   \sum_{i=1}^s A_i =\sum_{i=1}^s B_i=0. \label{eq:lin2-0}
\end{gather}
Let $X$ (resp.\ $Y$) be the set of rows of the OA, with symbols $0,1,\ldots, s-1$ replaced by $\pm A_1, \pm A_2$, $\ldots, \pm A_s$ (resp.\ $\pm B_1, \pm B_2, \ldots, \pm B_s$).
Then $(X,Y)$ is
a proper solution
% a nontrivial solution
of degree $m+3$ and size $2\ell$ for the $\PTE_r$.
\end{theorem}

\begin{remark} \label{rem:curious1}
If $s=m+1$ then the condition (\ref{eq:degtplus2-0}) follows from the conditions (\ref{eq:degt0}) and (\ref{eq:lin2-0}); see \cref{lin} for details.
\end{remark}

In our proof of \cref{thm:OAL1}, we make use of the Girard--Newton formulae and their consequence, where we write $p_k^{(\mathbf{z})}$ and $e_k^{(\mathbf{z})}$ for the power-sum and the elementary symmetric polynomial of degree $k$ in variables $z_1,\ldots,z_n$, respectively.

\begin{lemma} [{Girard--Newton formulae}] \label{lem:GN}  
It holds that for $k, n \in \mathbb{Z}_{>0}$,
\begin{align*}
p_k^{(\mathbf{z})}-e_1^{(\mathbf{z})}p_{k-1}^{(\mathbf{z})}+\cdots+ (-1)^{k-1}e_{k-1}^{(\mathbf{z})}p_1^{(\mathbf{z})}+(-1)^k ke_k^{(\mathbf{z})}=0 && (1 \leq k \leq n),\\
p_k^{(\mathbf{z})}-e_1^{(\mathbf{z})}p_{k-1}^{(\mathbf{z})}+\cdots+ (-1)^{n-1}e_{n-1}^{(\mathbf{z})}p_{k-n+1}^{(\mathbf{z})}+(-1)^n e_n^{(\mathbf{z})}p_{k-n}^{(\mathbf{z})}=0 && (k>n).
\end{align*}
\end{lemma}

\begin{lemma} [{\cite[Chapter~7, \S 1, Theorem~8]{CLO2025}}] \label{lem:sym}
Any symmetric polynomial $g(\mathbf{z}) = g(z_1,\ldots,z_n)$ is expressed as a polynomial of power sums $p_1^{(\mathbf{z})}, \ldots, p_n^{(\mathbf{z})}$.
\end{lemma}
\begin{proof}
By the fundamental theorem of symmetric polynomials, $g(\mathbf{z})$ is expressed as a polynomial of elementary symmetric polynomials $e_1^{(\mathbf{z})}, \ldots, e_n^{(\mathbf{z})}$.
By \cref{lem:GN}, $g(\mathbf{z})$ is further expressed as a polynomial of power sums $p_1^{(\mathbf{z})}, \ldots, p_n^{(\mathbf{z})}$.
\end{proof}

We are now ready to complete the proof of \cref{thm:OAL1}.

\begin{proof}[Proof of \cref{thm:OAL1}]
Since $X = -X$, $Y = -Y$, and $m$ is even, it is sufficient to show that
\begin{equation} \label{eq:lifting4-0}
\sum_{\mathbf{x} \in X} f(\mathbf{x}) = \sum_{\mathbf{y} \in Y} f(\mathbf{y}) \ \text{ for every monomial $f(\mathbf{z}) = z_1^{k_1} \cdots z_{t'}^{k_{t'}}$ of even degree $\le m+2$}.
\end{equation}
Without loss of generality, we may assume that $k_1,\ldots,k_{t'}$ are all positive.
It follows by the regularity condition of OA that for any $t' \le r$,
\begin{align*}
\sum_{\mathbf{x} \in X} f(\mathbf{x})
& = \lambda_{t'} \sum_{ (a_1,\ldots,a_{t'}) \in \{A_1,\ldots,A_s \}^{t'} } \left( \prod_{i=1}^{t'} (a_i^{k_i} + (-a_i)^{k_i}) \right) = 2\lambda_{t'} \sum_{ (a_1,\ldots,a_{t'}) \in \{A_1,\ldots,A_s \}^{t'}} \prod_{i=1}^{t'} a_i^{k_i}, \\
\sum_{\mathbf{y} \in Y} f(\mathbf{y})
& = \lambda_{t'}  \sum_{ (b_1,\ldots,b_{t'}) \in \{B_1,\ldots,B_s \}^{t'} } \left( \prod_{i=1}^{t'} (b_i^{k_i} + (-b_i)^{k_i}) \right) = 2\lambda_{t'} \sum_{ (b_1,\ldots,b_{t'}) \in \{B_1,\ldots,B_s \}^{t'} } \prod_{i=1}^{t'} b_i^{k_i},
\end{align*}
where $\lambda_{t'}$ is the notation defined in \cref{OAreg}.
Then there exists some homogeneous symmetric polynomial $g(z_1,\ldots,z_s)$ of degree $d$, where $d = \deg f$, such that
\[
g(A_1,\ldots,A_s) = \sum_{\mathbf{x} \in X} f(\mathbf{x}), \qquad g(B_1,\ldots,B_s) = \sum_{\mathbf{y} \in Y} f(\mathbf{y}).
\]
Let $\delta = \min\{d, s\}$.
By \cref{lem:sym}, $g(\mathbf{A})$ is expressed as a polynomial of power sums $p_1^{(\mathbf{A})},\ldots,p_{\delta}^{(\mathbf{A})}$, and similarly for $g(\mathbf{B})$.
If $d \le m$, then we obtain (\ref{eq:lifting4-0}) by the assumption (\ref{eq:degt0}).
If $d = m+2$, then the only term in $g(\mathbf{z})$ containing $p_{m+1}^{(\mathbf{z})}$ is $c p_1^{(\mathbf{z})} p_{m+1}^{(\mathbf{z})}$ where $c\in\Q$.
By the assumption (\ref{eq:lin2-0}), we obtain $c p_1^{(\mathbf{A})} p_{m+1}^{(\mathbf{A})} = c p_1^{(\mathbf{B})} p_{m+1}^{(\mathbf{B})} = 0$.
Hence, by the assumptions (\ref{eq:degt0}) and (\ref{eq:degtplus2-0}), we obtain (\ref{eq:lifting4-0}).

The properness of our solution follows by noting that $\rank X=\rank Y=r$, as in the second paragraph of the proof of \cref{OAtoPTE}.
\end{proof}

We now give a Type-I OA lifting, where we use the {\it regularity condition} for Type-I OA, namely the condition that for any $t'\le t$, an $\OA_I(\ell,r,s,t)_{\lambda}$ is an $\OA_I(\ell,r,s,t')_{\lambda_{t'}}$.

\begin{theorem} [Type-I OA lifting] \label{thm:OAL3} 
Let $m \ge 2$ be an even integer.
Assume that there exists an $\OA_I(\ell,r,s,s)_{\lambda}$ with $s \ge m+1$.
Moreover, assume that $\pm A_i, \pm B_i$ are mutually distinct rationals for which
\begin{gather}
  [A_1, \ldots, A_s] =_m^s [B_1, \ldots, B_s], \label{eq:degm} \\
   \sum_{i=1}^s A_i^{m+2}= \sum_{i=1}^s B_i^{m+2}, \label{eq:degmplus2} \\
   \sum_{i=1}^s A_i =\sum_{i=1}^s B_i=0. \label{eq:lin2}
\end{gather}
Let $X$ (resp.\ $Y$) be the set of rows of the Type-I OA, with symbols $0,1,\ldots, s-1$ replaced by $\pm A_1, \pm A_2, \ldots, \pm A_s$ (resp.\ $\pm B_1, \pm B_2, \ldots, \pm B_s$).
Then $(X,Y)$ is a solution of degree $m+3$ and size $2\ell$ for the $\PTE_r$.
\end{theorem}

\begin{proof}
Since $X = -X$, $Y = -Y$, and $t$ is even, it is sufficient to show that
\begin{equation} \label{eq:lifting4}
\sum_{\mathbf{x} \in X} f(\mathbf{x}) = \sum_{\mathbf{y} \in Y} f(\mathbf{y}) \ \text{ for every monomial $f(\mathbf{z}) = z_1^{k_1} \cdots z_{t'}^{k_{t'}}$ of even degree $\le m+2$}.
\end{equation}
Without loss of generality, we may assume that $k_1,\ldots,k_{t'}$ are all positive.
It follows by the regularity condition of Type-I OA that for any $t' \le s$,
\begin{align*}
\sum_{\mathbf{x} \in X} f(\mathbf{x})
&
= \lambda_{t'} \sum_{(a_1,\ldots,a_{t'}) \in \{A_1,\ldots,A_s \}^{(t')}} \left(\prod_{i=1}^{t'} (a_i^{k_i} + (-a_i)^{k_i}) \right)
= 2\lambda_{t'} \sum_{(a_1,\ldots,a_{t'}) \in \{A_1,\ldots,A_s \}^{(t')}} \prod_{i=1}^{t'} a_i^{k_i}, \\
\sum_{\mathbf{y} \in Y} f(\mathbf{y})
&
= \lambda_{t'}  \sum_{(b_1,\ldots,b_{t'}) \in \{B_1,\ldots,B_s \}^{(t')}} \left(\prod_{i=1}^{t'} (b_i^{k_i} + (-b_i)^{k_i}) \right)
= 2\lambda_{t'}  \sum_{(b_1,\ldots,b_{t'}) \in \{B_1,\ldots,B_s \}^{(t')}} \prod_{i=1}^{t'} b_i^{k_i},
\end{align*}
where $\{C_1,\ldots,C_s \}^{(t')}$ denotes the set of all $s!/(s-t')!$ permutations of $t'$ distinct symbols among $s$ distinct symbols $C_1,\ldots,C_s$.
Then, by using \cref{lem:sym} as in the proof of \cref{thm:OAL1}, we obtain the desired equation (\ref{eq:lifting4}).
\end{proof}

We now treat a modification of \cref{thm:OAL3}.
Note that the two-dimensional Borwein solution has the linearity condition (\ref{eq:linear0}).

\begin{theorem} [Three-dimensional Borwein solution] \label{thm:GeneralizedBorwein1}
Assume that there exists a disjoint pair of multisets $\{\pm  A_1, \pm A_2, \pm A_3\}, \{\pm B_1, \pm B_2, \pm B_3\} \subset \mathbb{Q}$ for which
\begin{equation} \label{PTE1deg1,2}
[A_1,A_2,A_3] =_2^3 [B_1,B_2,B_3],
\end{equation}
\begin{equation} \label{PTE1deg4}
A_1^4 + A_2^4 + A_3^4 = B_1^4 + B_2^4 + B_3^4, 
\end{equation}
\begin{equation} \label{eq:linear1}
A_1 + A_2 + A_3 = B_1 + B_2 + B_3 = 0. 
\end{equation}
Then it holds that
\begin{multline} \label{eq:GeneralizedBorwein1}
[\pm(A_1,A_2,A_3), \pm(A_2,A_3,A_1), \pm(A_3,A_1,A_2)] \\=_5^6 [\pm(B_1,B_2,B_3), \pm(B_2,B_3,B_1), \pm(B_3,B_1,B_2)].
\end{multline}
\end{theorem}
\begin{proof}
By the assumptions (\ref{PTE1deg1,2}), (\ref{PTE1deg4}), and (\ref{eq:linear1}), we have
\[
p_1^{(A_1,A_2,A_3)} = p_1^{(B_1,B_2,B_3)} = 0, \; p_2^{(A_1,A_2,A_3)} = p_2^{(B_1,B_2,B_3)}, \; p_4^{(A_1,A_2,A_3)} = p_4^{(B_1,B_2,B_3)}.
\]
We apply \cref{thm:OAL3} for $\OA_I(6,3,3,3)_1$ given in \cref{ex:OAII}.
Then we have
\small
\begin{multline} \label{eq:GeneralizedBorwein2}
[\pm(A_1,A_2,A_3), \pm(A_2,A_3,A_1), \pm(A_3,A_1,A_2), \pm(A_2,A_1,A_3), \pm(A_3,A_2,A_1), \pm(A_1,A_3,A_2)]
\\
=_5^{12} [\pm(B_1,B_2,B_3), \pm(B_2,B_3,B_1), \pm(B_3,B_1,B_2), \pm(B_2,B_1,B_3), \pm(B_3,B_2,B_1), \pm(B_1,B_3,B_2)].
\end{multline}
\normalsize
By using the linearity condition (\ref{eq:linear1}), we can verify that
\[
\begin{gathered}
\sum_{i=1}^3 A_i A_{i+1}^3 = \sum_{i=1}^3 A_i^3 A_{i+1}, \quad
\sum_{i=1}^3 B_i B_{i+1}^3 = \sum_{i=1}^3 B_i^3 B_{i+1}, \\
\sum_{i=1}^3 A_i A_{i+1} A_{i+2}^2 = \sum_{i=1}^3 A_i A_{i+1}^2 A_{i+2}, \quad
\sum_{i=1}^3 B_i B_{i+1} B_{i+2}^2 = \sum_{i=1}^3 B_i B_{i+1}^2 B_{i+2},
\end{gathered}
\]
where the subscripts of $A$ and $B$ are reduced cyclically modulo $3$.
This implies that
\[
\sum_{i=1}^3 (f(A_i,A_{i+1},A_{i+2}) + f(-A_i,-A_{i+1},-A_{i+2})) = \sum_{i=1}^3 (f(A_i,A_{i+2},A_{i+1}) + f(-A_i,-A_{i+2},-A_{i+1}))
\]
holds for all monomials $f$ of even degree $\le 4$, and eventually for all monomials of degree at most $5$.
Since the similar equation holds for $B_i$'s, we can reduce (\ref{eq:GeneralizedBorwein2}) to the desired equation  (\ref{eq:GeneralizedBorwein1}).
\end{proof}

\begin{remark} \label{rem:2DBorwein1}
The three-dimensional Borwein solution (\ref{eq:GeneralizedBorwein1}) is ideal but not combinatorially proper.
By restricting the first two coordinates of the six vectors in (\ref{eq:GeneralizedBorwein1}), we obtain
an ideal, combinatorially proper $\PTE_2$ solution as a formal generalization of the two-dimensional Borwein solution (\cref{2DBorwein}).
\end{remark}

%%%%%%%%%%%%%%%%%%%%%%%%%%%%%%%
%%%%%%%%%%%%%%%%%%%%%%%%%%%%%%%
\subsection{Cartesian product lifting and Jacroux's partitioning} \label{sect:cp_lift}

In this subsection, we present a lifting construction for the $\PTE_r$ that substantially differs from OA-based lifting developed in the previous subsection.
Our construction also generalizes a key lemma in Jacroux's work (\cite[Lemma 1]{Jacroux1995}) on the construction of sets of integers with equal power sums.

Given $S = \{ s_1,\ldots,s_{n_S} \} \subset \mathbb{Q}^{r_S}$ and $T = \{ t_1,\ldots,t_{n_T} \} \subset \mathbb{Q}^{r_T}$, we denote by $S \times T$ the Cartesian product of $S$ and $T$, that is
\[
S \times T = \bigcup_{i_S=1}^{n_S} \bigcup_{i_T=1}^{n_T} \{ (s_{i_S,1}, \ldots, s_{i_S,r_S}, t_{i_T,1}, \ldots, t_{i_T,r_T}) \}.
\]
We also use the notation $[\ell]=\{ 1, \ldots, \ell\}$. 

\begin{theorem}[Cartesian product lifting] \label{thm:cp_lift}
Let $\ell$ be an integer with $\ell \ge 2$.
For $a \in [\ell]$, define
\begin{align*}
\begin{split}
      S^{(a)} &= \{(s^{(a)}_{11}, \ldots, s^{(a)}_{1r_S}), \ldots, (s^{(a)}_{n_S1}, \ldots, s^{(a)}_{n_Sr_S})\} \subset \mathbb{Q}^{r_S},\\
      T^{(a)} &= \{(t^{(a)}_{11}, \ldots, t^{(a)}_{1r_T}), \ldots, (t^{(a)}_{n_T1}, \ldots, s^{(a)}_{n_Tr_T})\} \subset \mathbb{Q}^{r_T},
      \end{split}
       \end{align*}
and
\begin{align*} 
\begin{split}
         U^{(a)} 
         &\coloneqq \bigcup_{i=1}^{\ell} S^{(i)} \times T^{(\delta_{ai})} \\
         &=  \{(u^{(a)}_{11}, \ldots, u^{(a)}_{1,r_S+r_T}), \ldots, (u^{(a)}_{\ell n_Sn_T,1}, \ldots, u^{(a)}_{\ell n_Sn_T,r_S+r_T})\} \subset \mathbb{Q}^{r_S+r_T},
        \end{split}
\end{align*}
where $L = [\delta_{ij}]$ is a Latin square of order $\ell$.
Suppose that
\begin{align} \label{eq:ST_PTE}
 \text{$[S^{(a)}]=_{m_S}^{n_S} [S^{(b)}]$ and $[T^{(a)}]=_{m_T}^{n_T} [T^{(b)}]$ \; for any distinct $a,b\in[\ell]$}.
\end{align}
Then it holds that
\[
[U^{(a)}] =_{m_S+m_T+1}^{\ell n_Sn_T} [U^{(b)}] \; \text{ for any distinct $a,b\in[\ell]$}.
\]
\end{theorem}

\begin{proof}
Suppose that (\ref{eq:ST_PTE}) holds.
Let $e_1,\ldots,e_{r_S+r_T} \in \mathbb{Z}_{\ge0}$ such that $\sum_{i=1}^{r_S+r_T} e_i \le m_S+m_T+1$.
By the definition of $U^{(a)}$, we have
        \begin{align}\label{MPTE}
        \begin{split}
        \sum_{i=1}^{\ell n_Sn_T} \prod_{j=1}^{r_S+r_T} {(u^{(a)}_{ij})}^{e_j}
        &= \sum_{i=1}^{\ell}  \sum_{i_S=1}^{n_S} \sum_{i_T=1}^{n_T} \bigg(\prod_{j_S=1}^{r_S} \big( s^{(i)}_{i_Sj_S} \big)^{e_{j_S}} \bigg) \bigg(\prod_{j_T=1}^{r_T}
\Big( t^{(\delta_{ai})}_{i_Tj_T} \Big)^{e_{r_S+j_T}} \bigg)\\
        &= \sum_{i=1}^{\ell} \bigg( \sum_{i_S=1}^{n_S} \prod_{j_S=1}^{r_S} \big( s^{(i)}_{i_Sj_S} \big)^{e_{j_S}} \bigg) \bigg( \sum_{i_T=1}^{n_T} \prod_{j_T=1}^{r_T}
\big( t^{(\delta_{ai})}_{i_Tj_T} \big)^{e_{r_S+j_T}} \bigg).
         \end{split}
    \end{align}
By Dirichlet's box principle, we may assume, without loss of generality, that $\sum_{i=1}^{r_T} e_{r_S+i}\le m_T$.
Then the last term of (\ref{MPTE}) can be transformed as
    \begin{align*}
    \begin{split}
      & \sum_{i=1}^{\ell} \bigg( \sum_{i_S=1}^{n_S} \prod_{j_S=1}^{r_S} \big( s^{(i)}_{i_Sj_S} \big)^{e_{j_S}} \bigg) \bigg( \sum_{i_T=1}^{n_T} \prod_{j_T=1}^{r_T}
\big( t^{(\delta_{bi})}_{i_Tj_T} \big)^{e_{r_S+j_T}} \bigg)\\
        =& \sum_{i=1}^{\ell} \sum_{i_S=1}^{n_S} \sum_{i_T=1}^{n_T} \bigg(\prod_{j_S=1}^{r_S} \big( s^{(i)}_{i_Sj_S} \big)^{e_{j_S}} \bigg) \bigg(\prod_{j_T=1}^{r_T}
\Big( t^{(\delta_{bi})}_{i_Tj_T} \Big)^{e_{r_S+j_T}} \bigg) 
               =\sum_{i=1}^{\ell n_Sn_T} \prod_{j=1}^{r_S+r_T} {(u^{(b)}_{ij})}^{e_j}.
    \end{split}
    \end{align*}
    By the defining condition of Latin square, we have $U^{(a)} \cap U^{(b)} = \emptyset$. 
    Therefore, we obtain $[U^{(a)}]=_{m_S+m_T+1}^{\ell n_Sn_T} [U^{(b)}]$ for the $\PTE_{r_S+r_T}$.
\end{proof}

\begin{example}
[Cf. {\cite[Example 7]{Jacroux1995}}] \label{ex:Jacroux}
Let $S^{(1)}=T^{(1)}=\{1,6\}$, $S^{(2)}=T^{(2)}=\{2,5\}$, $S^{(3)}=T^{(3)}=\{3,4\}$.
Then it is obvious that
\[
[S^{(a)}] =_1^2 [S^{(b)}] \text{ and } [T^{(a)}] =_1^2 [T^{(b)}] \ \text{ for any distinct $a,b \in [3]$}.
\]
By using \cref{thm:cp_lift} with the Latin square $L_2$ given in \cref{ex:LS}, we obtain
$[U^{(a)}]=_3^{12} [U^{(b)}]$ for the $\PTE_2$, where
    \begin{align*}
          \{&U^{(1)},U^{(2)},U^{(3)} \}\\
            &=\{\{(1,1),(6,1),(1,6),(6,6),(2,3),(5,3),(2,4),(5,4),(3,2),(4,2),(3,5),(4,5)\},\\
            & \quad \{(1,2),(6,2),(1,5),(6,5),(2,1),(5,1),(2,6),(5,6),(3,3),(4,3),(3,4),(4,4)\},\\
            & \quad \{(1,3),(6,3),(1,4),(6,4),(2,2),(5,2),(2,5),(5,5),(3,1),(4,1),(3,6),(4,6)\} \}.
    \end{align*}
Note that this is a partition of the trivial $\OA(36,2,6,2)_1$ into three $\OA(12,2,6,1)_2$.
\end{example}

\cref{thm:cp_lift} is a high-dimensional extension of the following celebrated result by Jacroux:

\begin{corollary}
[{Jacroux's partitioning \cite[Lemma 1]{Jacroux1995}}] \label{cor:jac}
Let $S^{(1)}, \ldots, S^{(\alpha)}$ and $T^{(1)}$, $\ldots, T^{(\alpha)}$ be sets of positive integers 
such that $[S^{(1)}] =_{m_S}^{n_S} \cdots =_{m_S}^{n_S} [S^{(\alpha)}]$ and $[T^{(1)}] =_{m_T}^{n_T}  \cdots =_{m_T}^{n_T} [T^{(\alpha)}]$. 
Then it holds that
\[
[\overline{U}^{(1)}] =_{m_S+m_T+1}^{{\alpha} n_Sn_T}  \cdots =_{m_S+m_T+1}^{{\alpha} n_Sn_T} [\overline{U}^{({\alpha})}],\]
where
\[
\overline{U}^{(a)} = \bigcup_{i=1}^{\alpha} \bigcup_{\substack{t\in T^{(j)}\\i+j\equiv a+1 \; (\text{mod} \; \alpha)}} \{S^{(i)}+(t-1)\alpha n_S\} \ \text{ for $a\in[\alpha]$}
\]
and $S^{(i)}+(t-1)\alpha n_S$ is the set $S^{(i)}$ with $(t-1)\alpha n_S$ added to each element.
\end{corollary}

\begin{proof}
Let $a,b \in [\alpha]$ with $a \neq b$, and suppose that $[S^{(a)}]=_{m_S}^{n_S} [S^{(b)}]$ and $[T^{(a)}]=_{m_T}^{n_T} [T^{(b)}]$.
By applying \cref{thm:cp_lift} to them, we obtain a $\PTE_2$ solution given by
\[
[\{(u^{(a)}_{i1}, u^{(a)}_{i2})\}]=_{m_S+m_T+1}^{\alpha n_Sn_T} [\{(u^{(b)}_{i1}, u^{(b)}_{i2})\}].
\]
 It is then obvious that
\[
 [\{(u^{(a)}_{i1}, u^{(a)}_{i2},1)\}]=_{m_S+m_T+1}^{\alpha n_Sn_T} [\{(u^{(b)}_{i1}, u^{(b)}_{i2},1)\}]. 
\]  
Then by applying \cref{ts} to
\[
M=\begin{bmatrix}
1 & 0 & 0\\
\alpha n_S & 1 & 0\\
-\alpha n_S & 0 & 1
\end{bmatrix}
\in \GL_3(\mathbb{Q}),
\]
we obtain a $\PTE_3$ solution given by
\[
[(u^{(a)}_{i1}+(u^{(a)}_{i2}-1)\alpha n_S,u^{(a)}_{i2},1)]=_{m_S+m_T+1}^{\alpha n_Sn_T} [(u^{(b)}_{i1}+(u^{(b)}_{i2}-1)\alpha n_S,u^{(b)}_{i2},1)].
\]
This can be reduced  to the $\PTE_1$ solution
\[
[u^{(a)}_{i1}+(u^{(a)}_{i2}-1)\alpha n_S]=_{m_S+m_T+1}^{\alpha n_Sn_T} [u^{(b)}_{i1}+(u^{(b)}_{i2}-1)\alpha n_S].
\]
\end{proof}

\begin{example}
[{\cite[Example 7]{Jacroux1995}}]
Through the procedure described in our proof of \cref{cor:jac}, the sets $U^{(1)}$, $U^{(2)}$, $U^{(3)}$ given in \cref{ex:Jacroux} can be reduced to a partition of consecutive positive integers $1, \ldots, 36$ into
three sets of size $12$ with equal power sums as
\begin{align*}
\overline{U}^{(1)} &=\{1,6,31,36,14,17,20,23,9,10,27,28 \},\\
\overline{U}^{(2)} &=\{7,12,25,30,2,5,32,35,15,16,21,22 \},\\
\overline{U}^{(3)} &=\{13,18,19,24,8,11,26,29,3,4,33,34 \} .
\end{align*}
\end{example}

%%%%%%%%%%%%%%%%%%%%%%%%%%%%%%%%%%%%%%%%%%%%%%%%%%
%%%%%%%%%%%%%%%%%%%%%%%%%%%%%%%%%%%%%%%%%%%%%%%%%%
\section{A curious phenomenon involving ideal solutions: Half-integer designs} \label{sect:curious}

In this section, we prove a characterization theorem for ideal solutions of the $\PTE_1$.
Our motivation comes from a question of what is meant by the curious conditions $p_2^{(\mathbf{A})} = p_2^{(\mathbf{B})}, p_4^{(\mathbf{A})} = p_4^{(\mathbf{B})}$ and $p_1^{(\mathbf{A})} = p_1^{(\mathbf{B})} = 0$ that are essential in the proof of \cref{thm:GeneralizedBorwein1} (see also \cref{rem:curious1}).
The details are explained in the following main theorem of this section:

\begin{theorem} \label{lin}
Suppose that $n \ge 2$.
Let $x_1,\ldots, x_n , y_1, \ldots, y_n \in \mathbb{Q}$ and suppose that
\begin{align} \label{ideal}
[x_1, \ldots, x_n]=^n_{n-1} [y_1, \ldots, y_n].
\end{align}
Then the following are equivalent:
\begin{enumerate}[label=$(\roman*)$]
\item It holds that
\begin{align} \label{linear}
\sum_{i=1}^n x_i=0.
\end{align}
\item It holds that
\[
\sum_{i=1}^n x_i^{n+1} =\sum_{i=1}^n y_i^{n+1}.
\]
\end{enumerate}
\end{theorem}

To prove \cref{lin}, we use the notation $p_k^{(\mathbf{z})}$ and $e_k^{(\mathbf{z})}$ as used in the Girard-Newton formula (\cref{lem:GN}).

\begin{proof} [{Proof of \cref{lin}}]
First, we prove $(i) \Rightarrow (ii)$. Let
\begin{align*}
q^{(\mathbf{x})}(z):=\prod_{i=1}^n (z-x_i), && q^{(\mathbf{y})}(z):=\prod_{i=1}^n (z-y_i).
\end{align*}
Let $c_j:=(-1)^{n-j} e_{n-j}^{(\mathbf{x})}$ for $1 \leq j \leq n-1$ and $c_0^{(\mathbf{x})}:=(-1)^n e_n^{(\mathbf{x})}$.
Similarly, we define $c_0^{(\mathbf{y})}$.

By (\ref{ideal}), (\ref{linear}) and \cref{lem:GN},
we have $e_1^{(\mathbf{x})}=e_1^{(\mathbf{y})}=0$ and $e_j^{(\mathbf{x})}=e_j^{(\mathbf{y})}$ for $2 \leq j \leq n-1$. Then
\begin{align*}
q^{(\mathbf{x})}(z)&=z^n+c_{n-2} z^{n-2}+ \cdots +c_1 z+c_0^{(\mathbf{x})},\\
q^{(\mathbf{y})}(z)&=z^n+c_{n-2} z^{n-2}+ \cdots +c_1 z+c_0^{(\mathbf{y})}.
\end{align*}
Since $q^{(\mathbf{x})}(x_i)=0$ for all $i$, it follows from (\ref{linear}) that
\begin{align} \label{eqforx}
\begin{split}
0=\sum_{i=1}^n x_i q^{(\mathbf{x})} (x_i)&=\sum_{i=1}^n x_i^{n+1} +\sum_{j=1}^{n-2} c_j \sum_{i=1}^n x_i^{j+1}+ c_0^{(\mathbf{x})} \sum_{i=1}^n x_i\\
&=\sum_{i=1}^n x_i^{n+1} +\sum_{j=1}^{n-2} c_j \sum_{i=1}^n x_i^{j+1}.
\end{split}
\end{align}
Similarly, we have $q^{(\mathbf{y})}(y_i)=0$ for all $i$, and $\sum_{i=1}^n y_i=\sum_{i=1}^n x_i=0$ by (\ref{ideal}) and (\ref{linear}). Then it follows from (\ref{linear}) that
\begin{align} \label{eqfory}
\begin{split}
0=\sum_{i=1}^n y_i q^{(\mathbf{y})} (y_i)&=\sum_{i=1}^n y_i^{n+1} +\sum_{j=1}^{n-2} c_j \sum_{i=1}^n y_i^{j+1}+ c_0^{(\mathbf{y})} \sum_{i=1}^n y_i\\
&=\sum_{i=1}^n y_i^{n+1} +\sum_{j=1}^{n-2} c_j \sum_{i=1}^n y_i^{j+1}.
\end{split}
\end{align}
By subtracting (\ref{eqforx}) form (\ref{eqfory}), it follows from (\ref{ideal}) that 
\begin{align*}
0=\left(\sum_{i=1}^n x_i^{n+1}-\sum_{i=1}^n y_i^{n+1} \right)+\sum_{j=1}^{n-2}c_j \left(\sum_{i=1}^n x_i^{j+1}-\sum_{i=1}^n y_i^{j+1} \right)
=\sum_{i=1}^n x_i^{n+1}-\sum_{i=1}^n y_i^{n+1}.
\end{align*}

Next, we prove $(ii) \Rightarrow (i)$.
By assumption, we have $p_k^{(\mathbf{x})}=p_k^{(\mathbf{y})}$ for $k=1, \ldots, n-1, n+1$.
By \cref{lem:GN}, we have $e_k^{(\mathbf{x})}=e_k^{(\mathbf{y})}$ for $1 \leq k \leq n-1$.
If $e_n^{(\mathbf{x})}=e_n^{(\mathbf{y})}$, then we have $\{x_1,\ldots, x_n \}=\{y_1,\ldots, y_n \}$, which contradicts the disjointness of $x_i$'s and $y_i$'s.
Thus we have
 \begin{equation} \label{eneq}
 e_n^{(\mathbf{x})} \neq e_n^{(\mathbf{y})}.
 \end{equation}

By \cref{lem:GN} for $k=n+1$,
\begin{align*}
p_{n+1}^{(\mathbf{x})}-e_1^{(\mathbf{x})}p_n^{(\mathbf{x})}+\cdots+ (-1)^{n-1}e_{n-1}^{(\mathbf{x})}p_2^{(\mathbf{x})}+(-1)^n e_n^{(\mathbf{x})}p_1^{(\mathbf{x})}=0,\\
p_{n+1}^{(\mathbf{y})}-e_1^{(\mathbf{y})}p_n^{(\mathbf{y})}+\cdots+ (-1)^{n-1}e_{n-1}^{(\mathbf{y})}p_2^{(\mathbf{y})}+(-1)^n e_n^{(\mathbf{y})}p_1^{(\mathbf{y})}=0.
\end{align*}
Combining the above relations, we have
\begin{align*}
e_1^{(\mathbf{x})}p_n^{(\mathbf{x})}+(-1)^{n-1}e_n^{(\mathbf{x})}p_1^{(\mathbf{x})}=e_1^{(\mathbf{y})}p_n^{(\mathbf{y})}+(-1)^{n-1} e_n^{(\mathbf{y})}p_1^{(\mathbf{y})} =e_1^{(\mathbf{x})}p_n^{(\mathbf{y})}+(-1)^{n-1} e_n^{(\mathbf{y})}p_1^{(\mathbf{x})},
\end{align*}
which implies with $p_1^{(\mathbf{x})}=e_1^{(\mathbf{x})}$ that
\begin{align*} 
e_1^{(\mathbf{x})} \left(p_n^{(\mathbf{x})}+(-1)^{n-1}e_n^{(\mathbf{x})} \right)=e_1^{(\mathbf{x})} \left(p_n^{(\mathbf{y})}+(-1)^{n-1}e_n^{(\mathbf{y})} \right).
\end{align*}
To prove $e_1^{(\mathbf{x})}=0$, we show that $p_n^{(\mathbf{x})}+(-1)^{n-1}e_n^{(\mathbf{x})} \neq p_n^{(\mathbf{y})}+(-1)^{n-1}e_n^{(\mathbf{y})}$.
In fact,
by \cref{lem:GN} for $k=n$,
\begin{align*}
p_n^{(\mathbf{x})}-e_1^{(\mathbf{x})}p_{n-1}^{(\mathbf{x})}+\cdots+ (-1)^{n-1}e_{n-1}^{(\mathbf{x})}p_1^{(\mathbf{x})}+(-1)^n ne_n^{(\mathbf{x})}=0,\\
p_n^{(\mathbf{y})}-e_1^{(\mathbf{y})}p_{n-1}^{(\mathbf{y})}+\cdots+ (-1)^{n-1}e_{n-1}^{(\mathbf{y})}p_1^{(\mathbf{y})}+(-1)^n ne_n^{(\mathbf{y})}=0.
\end{align*}
Since $p_k^{(\mathbf{x})}=p_k^{(\mathbf{y})}$ and $e_k^{(\mathbf{x})}=e_k^{(\mathbf{y})}$ for $1 \leq k \leq n-1$,
\[
p_n ^{(\mathbf{x})}+(-1)^n ne_n^{(\mathbf{x})}=p_n ^{(\mathbf{y})}+(-1)^n ne_n^{(\mathbf{y})}.
\]
Now, suppose that
\[
p_n^{(\mathbf{x})}+(-1)^{n-1}e_n^{(\mathbf{x})}=p_n^{(\mathbf{y})}+(-1)^{n-1}e_n^{(\mathbf{y})}.
\]
Then $e_n^{(\mathbf{x})}=e_n^{(\mathbf{y})}$, which contradicts (\ref{eneq}). 
This completes the proof.
\end{proof}

\begin{remark} \label{linRem}
In the proof of \cref{lin},
the assumption (\ref{ideal}) is essential.
For example, the solution
\[
[1, 2, -3] =_1^3 [-4, 0, 4]
\]
is linear, i.e., $1+2+(-3)=-4+0+4=0$, but we have
\[
1^3+2^3+(-3)^3=-18 \neq 0=(-4)^3+0^3+4^3.
\]
\end{remark}

We close this section by mentioning the connection to a curious phenomenon that is sometimes reported for combinatorial designs or spherical designs with some kind of group-theoretic structures. 

\begin{definition} [{Cf.\ \cite{TTHS2025}}] \label{def:gd}
Let $\mu$ be a probability measure on a subset $\Omega$ of $\mathbb{R}^d$.
A nonempty finite subset $X$ of $\Omega$ is a {\it (geometric) $t$-design} or a {\it Chebyshev-type cubature of degree $t$} if
\begin{align} \label{cubature}
\frac{1}{|X|} \sum_{\mathbf{x} \in X} f(x) = \int_{\Omega} f(\mathbf{\omega}) d \mu \
\text{ for every polynomial $f$ of degree $k \le t$ on $\Omega$}.
\end{align}
In particular, a geometric $t$-design $X$ is called a {\it spherical $t$-design on $\mathbb{S}^{d-1}$} if $\Omega$ is the $(d-1)$-dimensional unit sphere $\mathbb{S}^{d-1} \subset \mathbb{R}^d$ with the uniform measure $\mu$.
\end{definition}

The following notion can be found in \cite[Definition 2]{Pache2005}; see \cite{BOT2015} for more general notion.

\begin{definition} [{Half-integer design (HID)}] \label{t+1/2} 
A design $X$ is said to have degree {\it $\ell+1/2$} if
(\ref{cubature}) holds for every integer $0 \le k \le \ell$, but not for $k = \ell+1$, and again holds for some $k = \ell' > \ell+1$.
\end{definition}

\begin{example} [{Cf.\ \cite{BS1981}, \cite[pp.\ 76--77]{HNT2023}}] 
\label{exam:half1}
The normalized $E_8$-root system $\frac{1}{\sqrt{2}} E_8$ is a spherical $(7+1/2)$-design over $\mathbb{S}^7$.
In addition, the set of minimal vectors of the Leech lattice $\frac{1}{2} \Lambda_{24}$ is a spherical $(11+1/2)$-design over $\mathbb{S}^{23}$.
\end{example}

An analogue of the HID phenomenon can also be observed in combinatorial design theory.

\begin{example} [{\cite[Example 63.5]{CD2007}}] \label{exam:half2}
Let $R = \{1,2,\ldots,12\}$ and
\begin{align*}
\mathcal{F} := & \{\{1,2,3,4,5,6\}, \{7,8,9,10,11,12\}, \{1,2,7,8\},
\{1,2,9,11\}, \{1,2,10,12\}, \{3,5,7,8\},\\
&  \{3,5,9,11\},  \{3,5,10,12\}, \{4,6,7,8\}, \{ 4,6,9,11 \}, \{4,6,10,12\} \}.
\end{align*}
We define
$\mathcal{B} = \bigcup_{B \in \mathcal{F}} \Orb_\mathfrak{G} (B)$, where $\mathfrak{G} := \langle (1 \; 2 \;  3 \; 4 \; 5)(7 \;  8 \;  9 \;  10 \;  11) \rangle$.
Then
$(R, \mathcal{B}, \{4,6\})$ is a $3$-$(12,\{4,6\},1)$.
This design is not ^^ $2$-balanced', since $|\{ B \in \mathcal{B} \mid \{1,2\} \subset B \}|=4$ while $|\{ B \in \mathcal{B} \mid \{1,7\} \subset B \}|=5$.
\end{example}

\begin{remark} \label{HID} 
As seen in \cref{exam:half1,exam:half2}, most of the HID phenomena reported so far have been observed under some kind of group-theoretic assumptions.
It is noteworthy that \cref{lin} does not assume any such group-theoretic structures and only supposes the existence of ideal solutions.
\end{remark}

%%%%%%%%%%%%%%%%%%%%%%%%%%%%%%%%%%%%%%
%%%%%%%%%%%%%%%%%%%%%%%%%%%%%%%%%%%%%%
\section{Concluding remarks and future works} \label{sect:conclusion}

This is the first paper that provides a new approach to the high-dimensional PTE problem through the connection with combinatorial design theory.
First in \Cref{sect:preliminary,sect:minimal}, we have reconsidered the definition of
proper solutions and given a combinatorial inequality for the size of such solutions (\cref{thm:NewBound}), together with tight examples which inherently have the structure of distinctive block $t$-designs or OA related to the LAT construction (\cref{thm:LAT0}). 
Next in \Cref{sect:design}, we have established novel criteria for the construction of $\PTE_r$ solutions via disjoint pairs of combinatorial designs; see  for example \cref{OAtoPTE,GDDtoPTE}.
In \Cref{sect:lifting}, we have developed two types of dimension-lifting constructions of $\PTE_r$ solutions in the presence of lower-dimensional PTE solutions or various combinatorial designs; see for example \cref{thm:OAL1,thm:cp_lift}. 
It is emphasized that our results offer a number-theoretic interpretation of some classical notions in combinatorial design theory, and moreover generalize many previous results, including the LAT construction (\cref{thm:LAT0}), Jacroux's work concerning sets of integers with equal power sums (\cref{cor:jac}), and the classical Borwein solution (\ref{Borwein}) and its two-dimensional extension (\cref{2DBorwein}).
Finally, in \Cref{sect:curious}, we have established a characterization theorem for ideal solutions of the classical $\PTE_1$ (\cref{lin}) and then discussed the relationship to geometric $(\ell+1/2)$-designs, curious phenomena that have been rarely reported in combinatorial design theory and spherical design theory (\cref{HID}).

We close this paper by suggesting future research directions.

Although we have proposed the combinatorial PTE inequality especially in the even-degree case, the basic idea for solutions of degree $2t$ can be formally applied to deriving solutions of degree $2t-1$.
However, the inequalities obtained by such a simple reduction seem to be far from the best, and so tight solutions seem to be never produced.
This leads us to the following problem:

\begin{problem} \label{prob:odd}
What is a ``natural" lower bound for the size of the $\PTE_r$ solutions of degree $2t-1$? Also, find tight solutions with respect to such lower bounds.
\end{problem}

Another research direction is to consider `multiplexing' the $\PTE_r$, as in the one-dimensional work by Jacroux (\cite{Jacroux1995}).
\begin{problem} [{Multiple $PTE_r$}] \label{prob:Prouhet}
The {\it $r$-dimensional multiple PTE problem of degree $m$, size $n$ and class $\alpha$}, asks whether there exists a disjoint collection of multisets
\begin{align*}  
X_j:= \{ (x_{11}^{(j)}, \ldots, x_{1r}^{(j)}), \ldots, (x_{n1}^{(j)}, \ldots, x_{nr}^{(j)}) \}\subset \mathbb{Z}^r \quad (1 \leq j \leq \alpha)
\end{align*}
such that
\begin{align}  \label{eq:Prouhet1}
[X_j]=_m^n [X_{j'}] \ \text{ for any distinct $j ,j' \in [\alpha]$.}
\end{align}
\end{problem}

The one-dimensional case goes back to Prouhet~\cite{Prouhet1851}, who found a solution of (\ref{eq:Prouhet1}) of degree $m$ and size $\alpha^m$ for $m$ and $\alpha$ in general.
Wright~\cite{Wright1949} proved that there exists a solution of degree $m$ and size $n = O(m^2)$.
His proof is nonconstructive, based on some techniques in extremal combinatorics; another  nonconstructive proof can be found in a recent paper by Nguyen~\cite{Nguyen2016}.

We remark that our \cref{OAtoPTE,GDDtoPTE,thm:cp_lift} are applicable to such multiple $\PTE_r$. 
Further discussion is beyond the scope of the present paper, to which we intend to return as a future work.

\bigskip

\noindent {\bf Acknowledgements.}
%%The authors would like to express their sincere appreciation to Akihiko Yukie for suggesting a novel research direction relating $\PTE_r$ and designs on discrete spaces. 
The authors would like to express their sincere appreciation to Akihiko Yukie for suggesting a novel research direction relating $\PTE_r$ and designs on discrete spaces, and for providing helpful comments on our earlier draft.
The third author would like to express his sincere gratitude to Akihiro Munemasa for pointing out that some of the techniques used in this paper are also applicable to \cref{prob:Prouhet}.
The authors
%%% The third author
would also like to thank Yoshinosuke Hirakawa, Masatake Hirao, Ryoma Saji and Shuji Yamamoto for fruitful discussion concerning the definition of proper solutions, and Eiichi Bannai for his valuable comments to the presentation of this paper.

%% \begin{bibdiv}
%% \begin{biblist}
%% \bibselect{quadrature}
%% \end{biblist}
%% \end{bibdiv} 

\end{document}